\documentclass[11pt]{article}
\usepackage{latexsym}
\usepackage{amssymb}
\usepackage{amsmath}
\usepackage{amsthm}
\theoremstyle{plain}
\newtheorem{thm}{Theorem}[section]
\newtheorem{lem}[thm]{Lemma}

\newtheorem{rem}[thm]{Remark}
\newtheorem{defin}[thm]{Definition}

\newtheorem{pro}[thm]{Proposition}

\newcommand{\RR}{{\mathbb R}}

\newcommand{\ep}{\varepsilon}

\newcommand{\tr}{{\rm tr}}

\numberwithin{equation}{section}
\begin{document}
\newcounter{aaa}
\newcounter{bbb}
\newcounter{ccc}
\newcounter{ddd}
\newcounter{eee}
\newcounter{xxx}
\newcounter{xvi}
\newcounter{x}
\setcounter{aaa}{1}
\setcounter{bbb}{2}
\setcounter{ccc}{3}
\setcounter{ddd}{4}
\setcounter{eee}{32}
\setcounter{xxx}{10}
\setcounter{xvi}{16}
\setcounter{x}{38}
\title
{Wong-Zakai approximation of solutions to reflecting stochastic differential
equations on domains in Euclidean spaces\footnote{A modified version will appear in
Stochastic Processes and their Applications, DOI: 10.1016/j.spa.2013.05.004}}
\author{Shigeki Aida\footnote{
This research was partially supported by
Grant-in-Aid for Scientific Research (B) No.~24340023.}
and Kosuke Sasaki\\
Mathematical Institute\\
Tohoku University,
Sendai, 980-8578, JAPAN\\
e-mail: aida@math.tohoku.ac.jp}
\date{}
\maketitle
\begin{abstract}
In this paper, we study the Wong-Zakai approximation
of the solution to the stochastic differential equation on a
domain $D$ in a Euclidean space with normal
reflection at the boundary. 
We prove the $L^p$ convergence of the approximation in
$C([0,T]\to \bar{D})$ under some general conditions on $D$.
\end{abstract}

\section{Introduction}
Stochastic differential equations(SDEs) are defined as 
stochastic integral equations.
The definition of the stochastic integrals is based on martingale theory 
although there are pathwise approaches to this problems
via rough path theory recently.
A simple relation between SDE and
usual ordinal differential equation(=ODE)
were found by Wong and Zakai~\cite{wong-zakai}.
That is, they consider Stratonovich SDE
and corresponding ODE which is obtained by
replacing the Brownian motion by the piecewise linear approximation
and prove that the solution of the ODE converges to the solution of
the Stratonovich SDE almost surely in the topology of uniform
convergence when the approximation becomes finer.
More general approximations of paths are found, {\it e.g.},
in \cite{ikeda-watanabe}.
When we consider SDE on a domain $D$ in $\RR^d$,
we need to consider boundary conditions.
In this paper, we study Wong-Zakai approximations of solutions 
to SDE with reflecting boundary conditions on $\bar{D}$ and 
prove the $L^p$ convergence of them to the solution in $C([0,T]\to \bar{D})$.
This is not a first study of Wong-Zakai approximation of
reflecting SDE.
Doss and Priouret~\cite{doss} proved the uniform convergence of
the Wong-Zakai approximations in probability 
in the case where $\partial D$ is sufficiently smooth.
Also Pettersson~\cite{pettersson} proved the almost sure convergence
in the case where $D$ is a convex domain with the property (B) in
Tanaka~\cite{tanaka} and the diffusion
coefficient is a constant matrix.
This result was improved by 
Ren and Xu~\cite{ren-xu1, ren-xu2}.
They studied Stroock-Varadhan's type support theorem for 
stochastic variational inequalities and showed the convergence in
probability in $C([0,T]\to \bar{D})$.
This result corresponds to the case of reflecting SDEs on convex
domains.
Actually the existence and uniqueness of 
the solutions were proved by
Lions and Sznitman~\cite{lions-sznitman} and Saisho~\cite{saisho}
for more general domains.
In such cases,
Evans and Stroock~\cite{evans-stroock} proved the weak convergence of
the law of the Wong-Zakai approximations.
Our results improve their weak convergence to $L^p$ convergence
in $C([0,T]\to \bar{D})$.
We note that there are studies of Euler and Euler-Peano approximations
of reflecting SDE.
We refer them to the papers \cite{slominski1, slominski2}
by S\l omi\'nski.

The paper is organized as follows.
In Section 2, we 
state our main theorem (Theorem~\ref{Main theorem}).
First, we recall the basic results on
the Skorohod problems and the existence and uniqueness of
the strong solutions of reflecting SDE
based on \cite{lions-sznitman}
and \cite{saisho}.
In particular, we explain the conditions
on domains under which we will work.
In Section 3, we prove $L^p$ convergence of
Euler-Peano approximations.
In Section 4, we prove our main theorem
by estimating the difference between the Euler-Peano 
and Wong-Zakai approximations.

\section{Preliminary and main theorem}
Let $D$ be a non-empty open connected set in $\RR^d$.
In this paper, we do not assume the boundedness of
the boundary of $D$ or $D$ itself.
We define the set ${\cal N}_x$ 
of inward unit normal vectors at the boundary
point $x\in \partial D$ by
\begin{align}
 {\cal N}_x&=\cup_{r>0}{\cal N}_{x,r}\\
{\cal N}_{x,r}&=\left\{{\mathbf n}\in \RR^d~|~|{\mathbf n}|=1,
 B(x-r{\mathbf n},r)\cap D=\emptyset\right\},
\end{align}
where $B(z,r)=\{y\in \RR^d~|~|y-z|<r\}$, $z\in \RR^d$, $r>0$.
In this paper, the function space $C^k_b$ denotes
a set of
$k$-times continuously differentiable functions
such that all their derivatives and themselves are
bounded.
Let us recall conditions (A), (B), (C)
following \cite{saisho}.

\begin{defin}
$(1)$~
Condition {\rm (A)} $(\mbox{uniform exterior sphere condition})$.
 There exists a constant
$r_0>0$ such that
\begin{align}
{\cal N}_x={\cal N}_{x,r_0}\ne \emptyset \quad \mbox{for any}~x\in
 \partial D.
\end{align}

\noindent
$(2)$~
Condition {\rm (B)}.
There exist constants $\delta>0$ and $\beta\ge 1$
satisfying:

for any $x\in\partial D$ there exists a unit vector $l_x$ such that
\begin{align}
 (l_x,{\mathbf n})\ge \frac{1}{\beta}
\qquad \mbox{for any}~{\mathbf n}\in 
\cup_{y\in B(x,\delta)\cap \partial D}{\cal N}_y.
\end{align}

\noindent
$(3)$~
Condition {\rm (C)}.\quad
There exists a $C^2_b$ function $f$ on $\RR^d$ 
and a positive constant $\gamma$ such that 
for any $x\in \partial D$, $y\in \bar{D}$, ${\mathbf n}\in {\cal N}_x$
it holds that
\begin{align}
 \left(y-x,{\mathbf n}\right)+\frac{1}{\gamma}\left((D
 f)(x),{\mathbf n}\right)|y-x|^2\ge 0.
\end{align}
\end{defin}

Note that if $D$ is a convex domain, the condition (A) holds for any
$r_0$ and the condition (C) holds
for $f\equiv 0$.
The admissibility condition on $D$ in \cite{lions-sznitman}
is the property that $D$ 
can be approximated by domains with smooth boundary in a certain
sense.
In this paper, we do not use such a property and we refer it
to \cite{lions-sznitman}.
Here we explain what 
Skorohod problem is.
Let
$w=w(t)$~$(0\le t\le T)$ be a continuous path on $\RR^d$
with $w(0)\in \bar{D}$.
The pair of paths $(\xi,\phi)$ on $\RR^d$
is a solution of a Skorohod problem
associated with $w$ if the following properties
hold.
\begin{itemize}
 \item[(i)] $\xi=\xi(t)$~$(0\le t\le T)$ is a continuous path 
in $\bar{D}$ with $\xi(0)=w(0)$.
\item[(ii)] It holds that $\xi(t)=w(t)+\phi(t)$
for all $0\le t\le T$.
\item[(iii)] $\phi=\phi(t)$~$(0\le t\le T)$ is a continuous bounded variation 
path on $\RR^d$ such that $\phi(0)=0$ and
\begin{align}
\phi(t)&=\int_0^t{\mathbf n}(s)d\|\phi\|_{[0,s]}\\
\|\phi\|_{[0,t]}&=\int_0^t1_{\partial D}(\xi(s))d\|\phi\|_{[0,s]}.
\end{align}
where ${\mathbf n}(t)\in {\cal N}_{\xi(t)}$ if 
$\xi(t)\in \partial D$.
\end{itemize}
In the above,
$\|\phi\|_{[0,t]}$ stands for the total variation norm
of $\phi$.
See (\ref{total variation}).

The existence and uniqueness of solutions were proved by
Tanaka~\cite{tanaka} for the convex domain with additional assumptions.
Lions and Sznitman proved the existence and uniqueness under
conditions (A), (B) and the admissibility of $D$.
This was proved without the addmissibility condition by 
Saisho\cite{saisho} as follows.

\begin{thm}\label{Existence of Skorohod problem}
 Assume conditions {\rm (A)} and {\rm (B)}.
Then there exists a unique solution to the Skorohod 
problem for any continuous path $w$.
Moreover the mapping $\Gamma : w\mapsto \xi$ is
continuous in the uniform convergence topology.
\end{thm}

 Doss and Priouret~\cite{doss} proved the convergence of
Wong-Zakai approximation.
They used the Lipschitz continuity of the Skorohod
map $\Gamma: w\mapsto \xi$ in the half space case.
Under conditions (A) and (B),
it is proved that 
$\Gamma$ is $1/2$-H\"older continuous map in the uniform
convergence topology.
See \cite{lions-sznitman, saisho}.
If $\Gamma$ is Lipschitz continuous, 
Doss and Priouret's approach may be 
applicable.
We use the notation
$L(w)=\Gamma(w)-w$ which corresponds to
the local time at the boundary $\partial D$.

The bounded variation norm of $\phi$ can be controlled by
the supremum norm of $w$ and the modulus of continuity.
Such an estimate is proved by Tanaka~\cite{tanaka}
in the case of convex domains.
Similar estimates are obtained by Saisho~\cite{saisho}
without assumptions of the convexity.
For our purpose, we need quantitative version of
Saisho's estimate.
To this end, we introduce the following quantities
of the continuous path $w$.
Let $0<\theta\le 1$ and define 
\begin{equation}
 \|w\|_{{\cal H},[s,t],\theta}=\sup_{s\le u<v\le t}
\frac{|w(v)-w(u)|}{|u-v|^{\theta}}.
\end{equation}
Also we use the oscillation and the total variation
of the path:
\begin{align}
\|w\|_{\infty,[s,t]}&=\max_{s\le u\le v\le t}|w(u)-w(v)|,\\
\|w\|_{[s,t]}&=\sup_{\Delta}\sum_{k=1}^N|w(t_k)-w(t_{k-1})|,
\label{total variation}
\end{align}
where $\Delta=\{s=t_0<\cdots<t_N=t\}$ is a partition 
of the interval $[s,t]$.

\begin{lem}\label{Estimate for phi}
Assume {\rm (A)} and {\rm (B)}.
Let $0<\theta\le 1$.
Then there exist positive constants $C_1,C_2,C_3$ which depend only on
$\theta$, $\delta$, $\beta$ and $r_0$ in the Assumptions {\rm (A)} and {\rm (B)}
such that 
\begin{align}
 \|\phi\|_{[s,t]}&\le
C_1\left(1+\|w\|_{{\cal H},[s,t],\theta}^{C_2}(t-s)\right)e^{C_3\|w\|_{\infty,[s,t]}}
\|w\|_{\infty,[s,t]}\qquad \mbox{for all}~0\le s<t\le T.
\end{align}
\end{lem}

\begin{proof}
Let $0\le s<t\le T$.
The proof of this lemma is essentially the same as
that of Proposition~3.1 in \cite{saisho}.
However, since the estimate in the above lemma is quantitative version
of Proposition~3.1, we give the proof for the sake of completeness and
reader's convenience.
First we define a sequence of times inductively by
\begin{align*}
T_0&=\inf\{u~|~\xi(u)\in \partial D, s\le u\le t\}, \\
t_n&=\inf\{u~|~|\xi(u)-\xi(T_{n-1})|\ge\delta, ~T_{n-1}<u\le t\},\qquad
 (n\ge 1)\\
T_n&=\inf\{u~|~\xi(u)\in\partial D,~ t_n\le u\le t\}.\qquad \quad
 \qquad\qquad (n\ge 1)
\end{align*}
We use the convention that 
the times are $t$ if the sets on the RHS are empty.
If $\xi(s)\notin \partial D$ and
$T_0=t$, $\xi$ does not hit $\partial D$
in the time interval $[s,t)$ and
$\|\phi\|_{[s,t]}=0$.
Hence it is sufficient to consider other cases.
In those cases,
since $\xi$ is a continuous path,
there exists a minimum natural number $N$ such that
$t=T_N$.
Let $T_{n-1}\le u<v\le T_{n}$~($1\le n\le N$).
We prove
\begin{align}
\|\phi\|_{[u,v]}&\le
\beta\left(\|\xi\|_{\infty,[u,v]}+\|w\|_{\infty,[u,v]}\right).\label{phi
 and xi}
\end{align}
Suppose $u,v\le t_n$.
Let $l=l_{\xi(T_{n-1})}$.
Then using the condition (B), we have
\begin{align*}
(l, \xi(v)-\xi(u))&=(l,w(v)-w(u))+(l,\phi(v)-\phi(u))\nonumber\\
&=(l,w(v)-w(u))+\int_u^v(l,{\mathbf n}(r))d\|\phi\|_{[u,r]}\nonumber\\
&\ge (l,w(v)-w(u))+\frac{1}{\beta}\|\phi\|_{[u,v]},
\end{align*}
which implies (\ref{phi and xi}).
Let us consider the case where $T_n>t_n$ and $v>t_n$.
Since $\|\phi\|_{[t_n,v]}=0$,
we obtain 
$$
\|\phi\|_{[u,v]}
=\|\phi\|_{[u,t_n]}\le\beta\left(\|\xi\|_{\infty,[u,t_n]}+\|w\|_{\infty,[u,t_n]}\right)
$$
which implies (\ref{phi and xi}).
By Lemma 2.3 (ii) in \cite{saisho}, for any $0\le s\le t\le T$,
\begin{align}
 |\xi(t)-\xi(s)|^2&\le
|w(t)-w(s)|^2+
\frac{1}{r_0}\int_{s}^t
|\xi(u)-\xi(s)|^2d\|\phi\|_{[0,u]}+
2\int_s^t\left(w(t)-w(u),d\phi(u)\right).\label{xit-xis}
\end{align}
Hence
\begin{align}
 |\xi(t)-\xi(s)|^2&\le
|w(t)-w(s)|^2+2\|w\|_{\infty,[s,t]}\|\phi\|_{[s,t]}+\frac{1}{r_0}
\int_s^t|\xi(u)-\xi(s)|^2d\|\phi\|_{[s,u]}.
\end{align}
By the Gronwall inequality (see Lemma~2.2 in \cite{saisho}), 
we obtain
\begin{align*}
 |\xi(t)-\xi(s)|^2&\le
\left(\|w\|_{\infty,[s,t]}^2+2\|w\|_{\infty,[s,t]}\|\phi\|_{[s,t]}\right)
\exp\left(\|\phi\|_{[s,t]}/{r_0}\right)\nonumber\\
&\le
\left\{\left(1+1/\ep^2\right)\|w\|_{\infty,[s,t]}^2+
\ep^2\|\phi\|_{[s,t]}^2\right\}\exp\left(\|\phi\|_{[s,t]}/{r_0}\right)
\end{align*}
and
\begin{align}
 |\xi(t)-\xi(s)|&\le
\left\{\left(1+1/\ep\right)\|w\|_{\infty,[s,t]}+
\ep\|\phi\|_{[s,t]}\right\}\exp\left(\|\phi\|_{[s,t]}/{2r_0}\right).
\label{xit-xis2}
\end{align}
Here $\ep$ is any positive number.
Now we prove that
for any $T_{n-1}\le u\le v\le T_n$,
\begin{align}
 \|\phi\|_{[u,v]}&\le
\beta\left(G(\|w\|_{\infty,[u,v]})+2\right)\|w\|_{\infty,[u,v]},
\label{estimate on phi}
\end{align}
where
\begin{align*}
 G(x)&=4\left\{1+\beta
 H(x)\right\}H(x),\\
H(x)&=\exp\left\{\beta\left(2\delta+x\right)/(2r_0)\right\}.
\end{align*}
We consider three cases
(i)~$T_{n-1}\le u<v\le t_n$, (ii)~$t_n\le u<v\le T_n$,
(iii)~$u\le t_n<v\le T_n$.
Let us consider the case (i).
In this case $\|\xi\|_{\infty,[u,v]}\le 2\delta$.
By combining this, (\ref{phi and xi}) and (\ref{xit-xis2}),
we have
\begin{align*}
 \|\xi\|_{\infty,[u,v]}&\le
\left\{\left(1+1/\ep\right)\|w\|_{\infty,[u,v]}+
\ep\beta\left(\|\xi\|_{\infty,[u,v]}+\|w\|_{\infty,[u,v]}\right)\right\}
H(\|w\|_{\infty,[u,v]}).
\end{align*}
Setting
$
 \ep=1/\left(2\beta H(\|w\|_{\infty,[u,v]})\right),
$
we obtain
\begin{align}
 \|\xi\|_{\infty,[u,v]}&\le
4\left(1+\beta H(\|w\|_{\infty,[u,v]})\right)
H(\|w\|_{\infty,[u,v]})\|w\|_{\infty,[u,v]}.
\end{align}
We consider the case (ii).
In this case, $\phi(r)=\phi(u)$ for all $u\le r\le v$ and
$\|\xi\|_{\infty,[u,v]}=\|w\|_{\infty,[u,v]}$.
Hence in the case of (iii),
\begin{align}
 \|\xi\|_{\infty,[u,v]}&\le
\|\xi\|_{\infty,[u,t_n]}+\|\xi\|_{\infty,[t_n,v]}\nonumber\\
&\le
4\left(1+\beta H(\|w\|_{\infty,[u,v]})\right)
H(\|w\|_{\infty,[u,v]})\|w\|_{\infty,[u,v]}+\|w\|_{\infty,[u,v]}.
\end{align}
Consequently, by (\ref{phi and xi}), the proof of (\ref{estimate on
 phi})
is finished.
Using (\ref{estimate on phi}),
\begin{align}
 \|\phi\|_{[s,t]}&\le
\beta\sum_{n=1}^{N}
\left(G(\|w\|_{\infty,[T_{n-1},T_n]})+2\right)\|w\|_{\infty,[T_{n-1},T_n]}
\nonumber\\
&\le N\beta(G(\|w\|_{\infty,[s,t]})+2)\|w\|_{\infty,[s,t]}.
\label{estimate for phi2}
\end{align}
We estimate $N$.
Suppose $N\ge 2$.
Since for any $1\le n\le N-1$,
\begin{align}
 \delta&=|\xi(T_{n-1})-\xi(t_n))|\nonumber\\
&\le
\|\xi\|_{\infty,[T_{n-1},T_n]}\nonumber\\
&\le
\left\{4\left(1+\beta H(\|w\|_{\infty,[T_{n-1},T_n]})\right)
H(\|w\|_{\infty,[T_{n-1},T_n]})+1\right\}
\|w\|_{\infty,[T_{n-1},T_n]}\nonumber\\
&\le
\left\{4\left(1+\beta H(\|w\|_{\infty,[T_{n-1},T_n]})\right)
H(\|w\|_{\infty,[T_{n-1},T_n]})+1\right\}
\|w\|_{{\cal H},[s,t],\theta}(T_n-T_{n-1})^{\theta}.
\end{align}
Thus we have
\begin{align*}
 T_n-T_{n-1}&\ge
\delta^{1/\theta}
\left[\left\{4\left(1+\beta H(\|w\|_{\infty,[T_{n-1},T_n]})\right)
H\left(\|w\|_{\infty,[T_{n-1},T_n]}\right)+1\right\}
\|w\|_{{\cal H},[s,t],\theta}\right]^{-1/\theta}.
\end{align*}
Summing the numbers on both sides from $n=1$ to $n=N-1$,
we obtain
\begin{align*}
 t-s&\ge
(N-1)
\delta^{1/\theta}
\left[\left\{4\left(1+\beta H(\|w\|_{\infty,[T_{n-1},T_n]})\right)
H\left(\|w\|_{\infty,[T_{n-1},T_n]}\right)+1\right\}
\|w\|_{{\cal H},[s,t],\theta}\right]^{-1/\theta}
\end{align*}
and
\begin{align}
 N-1&\le
\left[\delta^{-1}
\left\{4\left(1+\beta H(\|w\|_{\infty,[s,t]})\right)
H(\|w\|_{\infty,[s,t]})+1\right\}
\|w\|_{{\cal H},[s,t],\theta}\right]^{1/\theta}(t-s).
\label{estimate for N}
\end{align}
Clearly, this estimate is true when $N=1$.
The estimates (\ref{estimate for phi2}) and
(\ref{estimate for N}) complete the proof of the lemma.
\end{proof}

It is easy to see that the term
$\|w\|_{{\cal H},[s,t], \theta }$
in the above estimate 
can be replaced by a quantity defined by a 
modulus of continuity of $w$.
We emphasize that we just need the continuity of $w$
to estimate the bounded variation norm of $\phi$.
Also we note that this estimate is not sharp in
the sense that the quantity on the RHS
does not depend on the starting point $x$
although 
$\|\phi\|_{[s,t]}$ does.
If $w$ is a continuous bounded variation path,
we can prove the following estimate.
This estimate is used to prove the
exponential integrability of
$Y^N$ in the proof of Lemma~\ref{Moment estimate for XN}.

\begin{lem}\label{estimate for xi}
 Assume condition {\rm (A)} and
the existence of the solution $\xi$ to the Skorohod problem
for a continuous bounded variation path $w$.
Then the total variation of the solution $\xi$ has the estimate:
\begin{equation}
 \|\xi\|_{[s,t]}\le2(\sqrt{2}+1)\|w\|_{[s,t]}
\label{estimate for xi by w}
\end{equation}
\end{lem}

\begin{proof}
We write
\begin{equation*}
\omega(s,t)=\|w\|_{[s,t]},\quad
\eta_0(s,t)=|\xi(t)-\xi(s)|,\quad
\eta(s,t)=\|\xi\|_{[s,t]}.
\end{equation*}
We use
the estimate (\ref{xit-xis}).
Noting
\begin{align}
 \frac{1}{r_0}\int_{s}^t
\eta_0(s,u)^2d|\phi|_u&\le
\frac{1}{r_0}\left(
\int_s^t\eta(s,u)^2d_u\eta(s,u)+\int_s^t\eta(s,u)^2d_u\omega(s,u)\right)\nonumber\\
&\le \frac{1}{r_0}\left(\frac{1}{3}\eta(s,t)^3+\eta(s,t)^2\omega(s,t)\right)=:k(s,t)
\end{align}
and
\begin{align*}
 \left|2\int_s^t(w(t)-w(u),d\phi(u))\right|&\le
2\left(\int_s^t\omega(u,t)d_u\omega(s,u)+\int_s^t\omega(u,t)d_u\eta(s,u)\right)
\nonumber\\
&\le 2\left(\int_s^t(\omega(s,t)-\omega(s,u))d_u\omega(s,u)
+\int_s^t\omega(s,t)d_u\eta(s,u)\right)\nonumber\\
& =\omega(s,t)^2+2\omega(s,t)\eta(s,t),
\end{align*}
we obtain
\begin{equation}
 \eta_0(s,t)^2\le 2\omega(s,t)^2+2\omega(s,t)\eta(s,t)+k(s,t)
\end{equation}
and
\begin{equation*}
 2\eta_0(s,t)^2\le\eta_0(s,t)^2+\eta(s,t)^2\le
\omega(s,t)^2+(\omega(s,t)+\eta(s,t))^2+k(s,t).
\end{equation*}
Therefore we have
\begin{equation}
 \sqrt{2}\eta_0(s,t)\le 2\omega(s,t)+\eta(s,t)+\sqrt{k(s,t)}.
\end{equation}
Note that 
$$
\eta(s,t)=\lim_{|\Delta|\to 0}\sum_{i=1}^n\eta_0(t_{i-1},t_i),
$$
where $\Delta$ is a partition 
$s=t_0<\cdots<t_n=t$ and $|\Delta|=\sup_i(t_i-t_{i-1})$.
We consider the term $\sqrt{k(s,t)}$.
Using
\begin{align*}
\sqrt{k(s,t)}&\le
\sqrt{\eta(s,t)/(3r_0)}\eta(s,t)+
\sqrt{\omega(s,t)/(r_0)}\eta(s,t)
\end{align*}
and the additivity, $\eta(s,t)=\sum_{i=1}^n\eta(t_{i-1},t_i)$,
we obtain
\begin{align*}
 \sum_{i=1}^n\sqrt{k(t_{i-1},t_i)}&\le
\sup_i\left(
\sqrt{\eta(t_{i-1},t_i)/(3r_0)}+
\sqrt{\omega(t_{i-1},t_i)/(r_0)}
\right)\eta(s,t)\to 0\qquad \mbox{as $|\Delta|\to 0$}.
\end{align*}
Thus, we get
$\sqrt{2}\eta(s,t)\le 2\omega(s,t)+\eta(s,t)$ which proves 
the desired inequality.
\end{proof}

\begin{rem}
{\rm  Under the admissibility of the domain, Lions and Sznitman
proved that $\|\phi\|_{[s,t]}\le \|w\|_{[s,t]}$
which implies $\|\xi\|_{[s,t]}\le 2\|w\|_{[s,t]}$.
They use regularity property of the distance function
from the boundary $\partial D$.
So we may need some regularity condition on the boundary
to prove such a stronger estimate.
We note that there is a study of the regularity of
the distance function, {\it e.g.}, \cite{poliquin}.
However, the estimate (\ref{estimate for xi by w})
is enough for our purposes.}
\end{rem}

Let us recall the existence of strong solution and the 
uniqueness which is due to \cite{tanaka, lions-sznitman, saisho}.
Let $(\Omega,{\cal F},P)$ be a complete
probability space and ${\cal F}_t$ be the 
right-continuous filtration with the property
that ${\cal F}_t$ contains all null sets of
$(\Omega,{\cal F}, P)$.
Let $B=B(t)$ be an ${\cal F}_t$-Brownian motion on $\RR^n$.
Let $\sigma\in C(\RR^d\to \RR^n\otimes \RR^d)$,
$b\in C(\RR^d\to \RR^d)$ be continuous mappings.
We consider an SDE with reflecting boundary condition on $\bar{D}$:
\begin{align}
X(t)=x+\int_0^t\sigma(X(s))dB(s)+\int_0^tb(X(s))ds+
\Phi(t),\label{reflecting sde}
\end{align}
where $x\in \bar{D}$.
We denote this SDE by SDE$(\sigma,b)$ simply.
A pair of ${\cal F}_t$-adapted continuous processes $(X(t),\Phi(t))$ 
is called a solution to
(\ref{reflecting sde}) if the following holds.
Let
\begin{align}
Y(t)&=x+\int_0^t\sigma(X(s))dB(s)+\int_0^tb(X(s))ds 
\end{align}
Then $(X(\cdot,\omega),\Phi(\cdot,\omega))$ 
is a solution of the Skorohod problem
associated with $Y(\cdot,\omega)$ for 
almost all $\omega\in \Omega$.
The following result is due to \cite{saisho}.

\begin{thm}\label{Saisho}
 Assume $D$ satisfies conditions {\rm (A)} and {\rm (B)}
and $\sigma$ and $b$ are bounded and global Lipschitz maps.
Then there exists a unique strong solution to
$(\ref{reflecting sde})$.
\end{thm}

Here we note the following.
This follows from Garsia-Rodemich-Rumsey's estimate.

\begin{lem}\label{Garsia}
Let $F=F(t,\omega)$ be a $\RR^d$-valued continuous process with
the property that for all $p\ge 1$
\begin{align}
E[|F(t)-F(s)|^{2p}]&\le C_p|t-s|^p\qquad 0\le s\le t\le T.
\end{align}
Then for all $0<\theta<1$ and $p\ge 1$
there exist constants $C_{p,\theta}'$ which depends only on 
$C_p$ and $\theta$ such that
\begin{align}
 E[\|F\|_{{\cal H},[0,T],\theta/2}^p]&\le C_{p,\theta}'.
\end{align}
\end{lem}

If $\sigma$ is bounded, then
the quadratic variation of $M(t)=\int_0^t\sigma(X(s))dB(s)$
is bounded and 
we see the exponential integrability of
$\max_{0\le t\le T}|M(t)|$.
Therefore,
using Lemma~\ref{Estimate for phi} and
Lemma~\ref{Garsia} and Burkholder-Davis-Gundy's inequality,
we immediately obtain the following estimate.

\begin{lem}
 Assume the same assumptions as in Theorem~$\ref{Saisho}$.
 Let $p\ge 1$.
There exists a positive constant $C_p$ such that
\begin{align}
E[\|X\|_{\infty,[s,t]}^{2p}]&\le C_p|t-s|^p,\label{estimate for x}\\
 E[|\Phi\|_{[s,t]}^{2p}]&\le
C_p|t-s|^p.\label{estimate for phi}
\end{align}
\end{lem}

From now on, 
we always assume that
$\sigma$ belongs to $C^2_b$ and $b$
belongs to $C^1_b$.
Now, we are going to explain our main theorem.
Let $N\in {\mathbb N}$.
Let 
$X^N(t)$ be the solution to the reflecting ODE:
\begin{align}
 X^N(t)=x+\int_0^t\sigma(X^N(s))dB^{N}(s)+\int_0^tb(X^N(s))ds+
\Phi^N(t),\label{reflecting ode}
\end{align}
where
\begin{align}
 B^{N}(t)&=B(t_{k-1}^N)+\frac{\Delta_N B_k}{\Delta_N}(t-t_{k-1}^N)
\qquad t^N_{k-1}\le t\le t^N_k,\\
\Delta_N B_k&=B(t_{k}^N)-B(t^N_{k-1}),\qquad \Delta_N=T/N,\qquad
t_k^N=\frac{kT}{N}.
\end{align}
We already explained the existence of the strong solution 
to a reflecting SDE driven by a Brownian motion.
The definition of the solution to the above equation 
is similar to reflecting SDE.
The existence and uniqueness of the solutions follows
from Theorem~\ref{Saisho}.
We prove an existence and uniqueness 
theorem when the driving path is a
continuous bounded variation path
in Section~\ref{Proof of main theorem}.
The following is our main theorem.
In this paper, we do not intend to obtain the best order.
The order given below is probably far from best.

\begin{thm}\label{Main theorem}
 Assume {\rm (A)}, {\rm (B)} and {\rm (C)}.
Let $X$ be the solution to
{\rm SDE}$(\sigma,\tilde{b})$,
where
$\tilde{b}=b+\frac{1}{2}\tr (D\sigma)(\sigma)$.
Let $0<\theta<1$.
For any $p\ge 1$, there exists a positive constant
$C_{p,T,\theta}$ such that for all $N\in {\mathbb N}$,
\begin{align}
E\left[\max_{0\le t\le T}|X^N(t)-X(t)|^{2p}\right]&\le C_{p,T,\theta}
\Delta_N^{\theta/6}.
\end{align}
\end{thm}

As we noted, although this estimate may not be good,
by this result and Borel-Cantelli lemma, we can conclude 
\begin{align}
\lim_{N\to\infty}\max_{0\le t\le T} 
|X^{2^N}(t)-X(t)|=0
\qquad \mbox{almost surely.}
\end{align}
In order to prove this theorem,
we need the Euler-Peano approximation of
the solution.
We explain the Euler-Peano approximation in the next Section.

\section{Euler-Peano approximation}

In this section, we consider the Euler-Peano approximation 
$X^N_E$ of
$X$.
For $0\le k\le N$, set $t^N_k=kT/N$.
Let us define $X^N_E(t)$~$(0\le t\le T)$
as the solution to the Skorohod problem
inductively which is given by $X^N_E(0)=x\in \RR^d$ and
\begin{align}
 X^N_E(t)&=X^N_E(t^N_{k-1})+\sigma(X^N_E(t^N_{k-1}))(B(t)-B(t^N_{k-1}))+
b(X^N_E(t))(t-t^N_k)\nonumber\\
& \quad+\Phi^N_E(t)-\Phi^N_E(t^N_{k-1})\qquad t^N_{k-1}\le t\le t^N_k.
\end{align}
In other words,
$X^N_E$ satisfies
\begin{align}
 X^N_E(t)&=x+\int_0^t\sigma(X^N_E(\pi_N(s)))dB(s)+\int_0^tb(X^N_E(\pi_N(s)))ds+
\Phi^N_E(t),
\end{align}
where $\pi_N(t)=\max\{t^N_k~|~t^N_k\le t\}$.
Define
\begin{align}
 Y^N_E(t)&=
x+\int_0^{t}\sigma(X^N_E(\pi_N(s)))dB(s)
+\int_0^{t}b(X^N_E(\pi_N(s)))ds.
\end{align}
Then by the definition of the solution of the SDE,
it holds that
\begin{align}
 X^N_E(t)=\Gamma\left(Y^N_E\right)(t).\label{skorohod map for euler}
\end{align}

We prove

\begin{thm}\label{Euler-Peano approximation}
Assume {\rm (A)}, {\rm (B)} and {\rm (C)}.
Then for any $p\ge 1$,
there exists $C_{p}>0$ such that
\begin{align}
 E\left[\max_{0\le t\le T}|X^N_E(t)-X(t)|^{2p}\right]
\le C_{p}\Delta_N^{p}.
\end{align}
\end{thm}

This estimate was already proved 
in \cite{slominski1} for general convex domains under the conditions
that $\sigma$ and $b$
are bounded and global Lipschitz continuous.
Also the readers may find a result of local version of Euler-Peano
and Euler approximation under the conditions (A) and (B) only
in that paper.

To prove this theorem, we need the following lemma.

\begin{lem}\label{Moment estimate for XNE}
Assume {\rm (A)} and {\rm (B)}. Let $p\ge 1$.
There exists a positive constant $C_p$ which is independent of
$N$ such that
\begin{align}
E[\|X^N_E\|_{\infty,[s,t]}^{2p}]&
\le C_p|t-s|^p,\label{estimate for xe}\\
 E\left[\|\Phi^N_E\|_{[s,t]}^{2p}\right]&\le C_p|t-s|^p.\label{estimate for phie}
\end{align}
\end{lem}

\begin{proof}
 It suffices to prove (\ref{estimate for phie}).
Since $M^N_E(t)=\int_0^t\sigma(X^N_E(\pi_N(s)))dB(s)$ is a martingale whose
quadratic variation is uniformly bounded for $N$,
we see that
$$
\sup_NE[\exp(a\max_{0\le t\le T}|M^N_E(t)|)]<\infty
$$ for all
$a>0$.
Thus by
Lemma~\ref{Estimate for phi},
Lemma~\ref{Garsia}
and Burkholder-Davis-Gundy's inequality, we complete the proof.
\end{proof}

\begin{proof}[Proof of Theorem~$\ref{Euler-Peano approximation}$]
The following proof is a modification of
that of Lemma~3.1 in \cite{lions-sznitman}.
Note that we need just Lipschitz 
continuity of $\sigma$ and $b$ and their boundedness 
in the proof below.
It suffices to prove the case where $p\ge 2$.
 Define
\begin{align*}
Z^N(t)&=X^N_E(t)-X(t),\\
\mu_N(t)&=e^{-\frac{2}{\gamma}(f(X^N_E(t)+f(X(t))))},\\
k_N(t)&=\mu_N(t)|Z^N(t)|^2.
\end{align*}
Then we have
\begin{align}
\lefteqn{d k_N(t)}\nonumber\\
&=
\mu_N(t)\Biggl\{
2\left(Z^N(t),(\sigma(X^N_E(\pi_N(t)))
-\sigma(X(t)))dB(t)\right)\nonumber\\
& \quad +
2\left(Z^N(t),b(X^N_E(\pi_N(t)))-b(X(t))\right)dt\nonumber\\
& \quad +
\tr\left(({}^t\sigma\sigma)(X^N_E(\pi_N(t)))\right)dt
+\tr\left(({}^t\sigma\sigma)(X(t))\right)dt\nonumber\\
& \quad
-\tr\left(({}^t\sigma(X(t))\sigma(X^N_E(\pi_N(t))))\right)
-\tr\left(({}^t\sigma(X^N_E(\pi_N(t)))\sigma(X(t)))\right)
dt\Biggr\}\nonumber\\
&\quad +
2\mu_N(t)\left(Z^N(t),d\Phi^N_E(t)-d\Phi(t)\right)\nonumber\\
&\quad
-\frac{2\mu_N(t)}{\gamma}
\left|Z^N(t)\right|^2\Bigl\{
\left((D f)(X^N_E(t)),d\Phi^N_E(t)\right)
+
\left((D f)(X(t)),d\Phi(t)\right)
\Bigr\}\nonumber\\
&\quad-\frac{2\mu_N(t)}{\gamma}
\left|Z^N(t)\right|^2\Bigl\{
\left((D f)(X^N_E(t)),\sigma(X^N_E(\pi_N(t)))dB(t)\right)
+
\left((D f)(X(t)),\sigma(X(t))dB(t)\right)
\Bigr\}\nonumber\\
&\quad +R_N(t)dt, \label{kNt}
\end{align}
where
\begin{align}
R(t)
&=\frac{4\mu_N(t)}{\gamma}
\left((D f)(X^N_E(t)),\sigma(X^N_E(\pi_N(t)))
{}^t\left(\sigma(X(t))-\sigma(X^N_E(\pi_N(t)))\right)
\left(Z^N(t)\right)\right)dt\nonumber\\
& \quad +\frac{4\mu_N(t)}{\gamma}
\left((D f)(X(t)),\sigma(X(t))
{}^t\left(\sigma(X^N_E(\pi_N(t)))-\sigma(X(t))\right)
\left(Z^N(t)\right)\right)dt\nonumber\\
& \quad -\frac{2\mu_N(t)}{\gamma}|Z^N(t)|^2
\left(\left((D f)(X^N_E(t)),{b}(X^N_E(\pi_N(t)))\right)dt+
\left((D f)(X(t)),{b}(X(t))\right)dt
\right)\nonumber\\
& \quad -\frac{\mu_N(t)}{\gamma}|Z^N(t)|^2
\Bigl\{\tr(D^2f)(X^N_E(t))(\sigma(X^N_E(\pi_N(t))\cdot,
\sigma(X^N_E(\pi_N(t)))))
\nonumber\\
& \quad \quad\quad \quad\qquad\quad +
\tr(D^2f)(X(t))(\sigma(X(t)\cdot,\sigma(X(t))\cdot))\Bigr\}dt
\nonumber\\
& \quad +\frac{2\mu_N(t)}{\gamma^2}
|(D f)(X^N_E(t))(\sigma(X^N_E(\pi_N(t))))
+(D f)(X(t))(\sigma(X(t)))
|^2
|Z^N(t)|^2dt.
\end{align}
Note that by condition (C),
\begin{align}
\lefteqn{\left(X^N_E(t)-X(t),d\Phi^N_E(t)-d\Phi(t)\right)}\nonumber\\
&\qquad
-\frac{1}{\gamma}
\left|X^N_E(t)-X(t)\right|^2\Bigl\{
\left((D f)(X^N_E(t)),d\Phi^N_E(t)\right)
+
\left((D f)(X(t)),d\Phi(t)\right)
\Bigr\} \le 0
\end{align}
and $\sup_{0\le t\le T}E[|X^N_E(t)-X^N_E(\pi_N(t))|^{p}]\le
 C\Delta_N^{p/2}$.
As for the first term on the RHS of (\ref{kNt}),
using Bukholder-Davis-Gundy's inequality,
we get for any $0\le T'\le T$,
\begin{align}
\lefteqn{E\left[\sup_{0\le t\le T'}
\left|\int_0^t\mu_N(s)
\left(X^N_E(s)-X(s),
\sigma(X^N_E(\pi_N(s)))-\sigma(X(s))dB(s)
\right)\right|^p\right]}\nonumber\\
&\le
CE\left[\left(\int_0^{T'}|X^N_E(t)-X(t)|^4dt\right)^{p/2}\right]
+CE\left[
\left(\int_0^{T'}
|X^N_E(\pi_N(t))-X^N_E(t)|^4dt\right)^{p/2}
\right]\nonumber\\
&\le C_TE\left[\int_0^{T'}k_N(t)^pdt\right]+
C_TE\left[\int_0^{T'}|X^N_E(\pi_N(t))-X^N_E(t)|^{2p}dt\right]\nonumber\\
&\le C_T\int_0^{T'}E[k_N(t)^p]dt+
C_T\Delta_N^p.
\end{align}
We can estimate the other terms similarly and
we obtain
\begin{align}
E\left[\sup_{0\le t\le T'}k_N(t)^p\right]
&\le C_T\Delta_N^p+C_T\int_0^{T'}E\left[\sup_{0\le s\le t}k_N(s)^p\right]dt.
\end{align}
By the Gronwall inequality,
this implies the desired estimate.
\end{proof}

\section{Proof of main theorem}\label{Proof of main theorem}

First, we prove the existence and uniqueness
of the solution to reflecting ODE driven by a
continuous bounded variation path.

\begin{pro}\label{Reflecting ODE}
Assume the conditions {\rm (A)} and {\rm (B)} hold.
 Let $w=w(t)$~$(0\le t\le T)$ be a continuous bounded variation path on $\RR^n$.
Then there exists a unique continuous bounded variation 
path $x(t)$ on $\RR^d$ satisfying the reflecting ODE:
\begin{align}
 x(t)&=x+\int_0^t\sigma(x(s))dw(s)+\int_0^tb(x(s))ds+\Phi(t),
\qquad 0\le t\le T.
\end{align}
\end{pro}

\begin{proof}
The following proof is a modification of
the proof of Theorem 5.1 in \cite{saisho}.
Note that the boundedness and the continuity of
$\sigma$ and $b$ are sufficient for the existence of the solutions.
Let us consider the partition of
$[0,T]$ by
$t^N_k=kT/N$.
 Let $x^N$ be the Euler-Peano approximation of the solution,
that is, let us define $x^N$ as the solution of
the Skorohod problem with $x^N(0)=x$:
\begin{align}
 x^N(t)&=x^N(t^N_{k-1})+\sigma(x^N(t^N_{k-1}))(w(t)-w(t^N_{k-1}))\nonumber\\
&\qquad +b(x^N(t^N_{k-1}))(t-t^N_{k-1})+\Phi^N(t)-\Phi^N(t^N_{k-1})
\qquad t^N_{k-1}\le t\le t^N_k.
\end{align}
Let
\begin{align}
 y^N(t)&=x+\int_0^t\sigma(x^N(\pi_N(s))dw(s)+\int_0^tb(x^N(\pi_N(s)))ds
\label{reflecting ode for yN}
\end{align}
Then $\{y^N\}$ is a family of uniformly bounded equicontinuous paths 
defined on $[0,T]$ with values in $\RR^d$.
Therefore by the Arzela-Ascoli theorem,
there exists a subsequence $\{y^{N_k}\}$ which converges in
the uniform convergence topology.
We denote the limit by $y^{\infty}$.
Then by the continuity of the Skorohod map in 
Theorem~\ref{Existence of Skorohod problem},
$x^{N_k}(=\Gamma(y^{N_k}))$, $\Phi^{N_k}(=L(y^{N_k})$
also converges to a continuous paths, say,
$x^{\infty}$, $\Phi^{\infty}$, in uniform convergence topology.
Clearly, the pair
$(x^{\infty},\Phi^{\infty})$ is a solution of
a Skorohod problem associated with $y^{\infty}$.
Taking the limit $N_{k}\to\infty$ in (\ref{reflecting ode for yN}),
we have
\begin{align}
 y^{\infty}(t)&=x+\int_0^t\sigma(x^{\infty}(s))dw(s)+\int_0^tb(x^{\infty}(s)))ds.
\label{reflecting ode for yinfty}
\end{align}
This shows that $(x^{\infty}, \Phi^{\infty})$
is a solution of the reflecting ODE.
We can check the uniqueness in a similar manner to 
Theorem 5.1 in \cite{saisho}.
Note that the boundedness of $\sigma$ and $b$ and
their Lipschitz continuity are sufficient for the proof.
\end{proof}

\begin{rem}{\rm
 We may prove the existence of the solution of
reflecting ODE when the driving path is just $p$-variation path,
where $1\le p<2$ using Davie's argument~\cite{davie}.
We will study this problem hopefully together with more general rough
differential equation corresponding to the case of
$p\ge 2$ in future's paper.}
\end{rem}
From now on, for simplicity, we may denote
$\Delta_N B_k$, $\Delta_N$, $t^N_k$ by
$\Delta B_k$, $\Delta$, $t_k$.
By the definition, it holds that
\begin{align}
 X^N(t)&=X^N(t_{k-1})+\int_{t_{k-1}}^t\sigma(X^N(s))\frac{\Delta
 B_k}{\Delta}ds+
\int_{t_{k-1}}^tb(X^N(s))ds\\
& \quad +\Phi^N(t)-\Phi^N(t_{k-1})
\quad \quad t_{k-1}\le t\le t_{k}.
\end{align}
Clearly, $X^N(t_{k-1})$ is
${\cal F}_{t_{k-1}}$-measurable.
Let
\begin{equation}
 Y^N(t)=x+\int_0^t\sigma(X^N(s))dB^{N}(s)+\int_0^tb(X^N(s))ds.
\end{equation}
Then $X^{N}=\Gamma(Y^{N})$ and $\Phi^N=L(Y^N)$.

\begin{lem}\label{Estimate for YN}
Assume {\rm (A)} and {\rm (B)}.
 Fix $N\in {\mathbb N}$.
Let $t_{k-1}\le s\le t\le t_k$.
The constant $C$ below is independent of
$t,s,k,N$.

\noindent
$(1)$ The following relations hold.
\begin{align}
Y^N(t)-Y^N(t_{k-1})
&=
\int_{t_{k-1}}^t\sigma(X^N(s))\frac{\Delta
 B_k}{\Delta}ds+
\int_{t_{k-1}}^tb(X^N(s))ds\label{YNt-YN0}
\end{align}
and
\begin{align}
 |Y^N(t)-Y^N(s)|&\le C\left(|\Delta B_k|\frac{t-s}{\Delta}+t-s\right)
\label{estimate for YN}\\
\|\Phi^N\|_{[s,t]}&\le C\left(|\Delta B_k|\frac{t-s}{\Delta}+t-s\right).
\label{estimate for PhiN}
\end{align}

\noindent
$(2)$ We have
\begin{align}
\int_{t_{k-1}}^t\sigma(X^N(s))\frac{\Delta
 B_k}{\Delta}ds
&=\sigma(X^N(t_{k-1}))\frac{\Delta B_k}{\Delta}(t-t_{k-1})\nonumber\\
&\quad +
\int_{t_{k-1}}^t\left(\int_{t_{k-1}}^s(D\sigma)(X^N(r))\sigma(X^N(r))
\frac{\Delta B_k}{\Delta}dr\right)\frac{\Delta B_k}{\Delta}ds\nonumber\\
&\quad
 +\int_{t_{k-1}}^t\left(\int_{t_{k-1}}^s(D\sigma)(X^N(r))(b(X^N(r)))dr\right)
\frac{\Delta B_k}{\Delta}ds\nonumber\\
&\quad
+\int_{t_{k-1}}^t\left(\int_{t_{k-1}}^s(D\sigma)(X^N(r))d\Phi^N(r)\right)
\frac{\Delta B_k}{\Delta}
ds\nonumber\\
&=I_0^k(t)+I_1^k(t)+I_2^k(t)
+I_3^k(t).\label{YNt-YN}
\end{align}
Let $I_4^k(t)=\int_{t_{k-1}}^tb(X^N(s))ds$.
Then
\begin{align}
 |I_1^k(t)|&\le C |\Delta B_k|^2
\frac{(t-t_{k-1})^2}{\Delta^2},\label{I1k}\\
|I_2^k(t)|&\le C|\Delta B_k|\frac{(t-t_{k-1})^2}{\Delta},\label{I2k}\\
|I_3^k(t)|&\le C\left(|\Delta B_k|^2\left(\frac{t-t_{k-1}}{\Delta}\right)^2+
\frac{(t-t_{k-1})^2}{\Delta}|\Delta B_k|\right),\label{I3k}\\
|I_4^k(t)|&\le C(t-t_{k-1}).\label{I4k}
\end{align}
\end{lem}

\begin{proof}
The proof of the equation (\ref{YNt-YN0}) and
(\ref{YNt-YN}) is a simple calculation.
The estimate in (\ref{estimate for YN}) follows from 
(\ref{YNt-YN0}).
Hence the estimate (\ref{estimate for PhiN})
follows from this estimate and Lemma~\ref{estimate for xi}.
By the boundedness of $\sigma,D\sigma, b$, we get
(\ref{I1k}), (\ref{I2k}), (\ref{I4k}).
Using (\ref{estimate for PhiN}),
\begin{align}
 |I_3^k(t)|&\le
C\|\Phi^N\|_{[t_{k-1},t]}\frac{(t-t_{k-1})|\Delta
 B_k|}{\Delta}\nonumber\\
&\le C\left(|\Delta B_k|^2\left(\frac{t-t_{k-1}}{\Delta}\right)^2+
\frac{(t-t_{k-1})^2}{\Delta}|\Delta B_k|\right).
\end{align}
This complete the proof.
\end{proof}

\begin{lem}\label{Osc for YN}
Assume {\rm (A)} and {\rm (B)}.
Let $p\ge 1$.
There exists a positive constant $C_p$ which is independent of
$N$ such that for all $0\le s\le t\le T$,
\begin{align}
 E[\|Y^N\|_{\infty,[s,t]}^{2p}]\le C_p|t-s|^p.
\end{align}
\end{lem}

\begin{proof}
Pick two points $0\le s\le t\le T$.
First consider the case where there exists
$1\le k\le N$ such that $t_{k-1}\le s\le t\le t_k$.
Then by (\ref{estimate for YN}),
\begin{align}
 \max_{s\le u\le v\le t}|Y^N(u)-Y^N(v)|
\le C(|\Delta B_k|\frac{t-s}{\Delta}+t-s).
\end{align}
Hence $E[\|Y^N\|_{\infty,[s,t]}\|^{2p}]\le C_p(t-s)^p$.
If $t_{k-1}\le s\le t_k<t\le t_{k+1}$ for some $k$,
noting
\begin{align}
\|Y^N\|_{\infty,[s,t]} \le \|Y^N\|_{\infty,[s,t_k]}+\|Y^N\|_{\infty,[t_k,t]},
\end{align}
we can use the estimate in the first case.
We consider the other cases.
Let us choose $1\le l<m-1\le N$ such that 
$t_{l-1}\le s\le t_l<t_{m-1}\le t\le t_m$.
Then
\begin{align}
\lefteqn{Y^N(t)-Y^N(s)}\nonumber\\
&=
\sum_{n=0}^4\left\{
I_n^l(t_l)-I_n^l(s)+\sum_{k=l+1}^{m-1}(I^k_n(t_k)-I^k_n(t_{k-1}))+
I^m_n(t)-I^m_n(t_{m-1})\right\}\label{YN-YN}\\
&=\sum_{n=0}^4(J_n^N(t)-J_n^N(s)).\nonumber
\end{align}
Note that $\{J_n^N(t)~|~0\le t\le T\}$ 
are continuous processes and it suffices to estimate
$E[\|J_n^N\|_{\infty,[s,t]}^{2p}]$.
First let us consider the term $J_0^N$.
Let $M^N(t)$ be a continuous ${\cal F}_t$-martingale such that
\begin{align}
 M^N(t)&=\int_0^t\sigma(X^N(\pi_N(s)))dB(s).\label{martingale MN}
\end{align}
Then $J^N_0$ is the piecewise linear approximation
of $M^N$ at the times $\{t_k\}_{k=1}^N$.
Therefore,
\begin{align}
 \|J^N_0\|_{\infty,[s,t]}&\le \max_{l-1\le k,k'\le
 m}|M^N(t_k)-M^N(t_{k'})|\nonumber\\
&\le 2\max_{l-1\le k\le m}|M^N(t_k)-M^N(t_l)|\nonumber\\
&\le 2\max_{t_{l-1}\le r\le t_m}|M^N(r)-M^N(t_l)|.\label{oscJN0}
\end{align}
Using Doob's inequality, we get
\begin{align}
 E[\|J_0^N\|_{\infty,[s,t]}^{2p}]\le C_p(t_m-t_{l-1})^{p}\le 3^pC_p(t-s)^p.
\end{align}
Next we consider the term $J^N_3$.
By the estimate (\ref{I3k}), we have
\begin{align}
\|J^N_3\|_{\infty,[s,t]}&\le
C\sum_{k=l}^m\left(
|\Delta B_k|^2+\Delta\cdot |\Delta B_k|\right)\le
C\left(\sum_{k=1}^m|\Delta B_k|^2\right)+C\Delta(t-s).
\end{align}
Note that
\begin{align}
\{\Delta B_k\}_{k=l}^m=\sqrt{\Delta}\{\xi_k\}_{k=l}^m\quad \mbox{in law},
\end{align}
where $\{\xi_k\}_{k=l}^m$ are i.i.d. random vectors
whose common distribution is the normal distribution 
on $\RR^n$ with $0$ mean and identity covariance matrix.
Hence
\begin{align}
E\left[\|J^N_3\|_{\infty,[s,t]}^{2p}\right]
&\le
 C_p\Delta^{2p}E\left[\left(\sum_{k=l}^m|\xi_k|^2\right)^{2p}\right]
+C_p(t-s)^{4p}.
\end{align}
Since $S_{m,l}=\sum_{k=l}^m(|\xi_k|^2-n)$ belongs to the Wiener chaos 
of order 2, there exists a constant $C_q$~$(q\ge 1)$ 
which is independent of
$m,l$ such that
\begin{align}
\|S_{m,l}\|_{L^{q}}\le C_q\|S_{m,l}\|_{L^2}.
\end{align}
This follows from the hypercontractivity of the Ornstein-Uhlenbeck
operator.
See \cite{bogachev}.
Therefore
\begin{align}
E\left[\left(\sum_{k=l}^m|\xi_k|^2\right)^{2p}\right]&=
E\left[\left\{S_{m,l}+n(m-l+1)\right\}^{2p}\right]\nonumber\\
&\le C_p\|S_{m,l}\|_{L^2}^{2p}+C_p\left\{n(m-l+1)\right\}^{2p}\nonumber\\
&\le C_p\left\{n(m-l+1)\right\}^p+C_p\left\{n(m-l+1)\right\}^{2p}.
\end{align}
Thus
\begin{align}
 E\left[\|J^N_3\|_{\infty,[s,t]}^{2p}\right]&\le
C_pn^{p}\Delta^p(t-s)^p+C_pn^{2p}(t-s)^{2p}+C_p(t-s)^{4p}\nonumber\\
&\le C_p (t-s)^{2p}.
\end{align}
We can estimate other terms in a similar way and
we complete the proof.
\end{proof}

\begin{lem}\label{Moment estimate for XN}
Assume {\rm (A)} and {\rm (B)}.
Let $p\ge 1$.
There exists a positive number $C_p$ which is independent of
$N$ such that for all $0\le s\le t\le T$,
\begin{align}
E[\|X^N\|_{\infty,[s,t]}^{2p}]&\le C_p|t-s|^p,\\
E[|\Phi^N\|_{[s,t]}^{2p}]&\le C_p|t-s|^p.\label{phiN}
\end{align}
\end{lem}

\begin{proof}
 It suffices to prove (\ref{phiN}).
By checking the exponential integrability of
$\|Y^N\|_{\infty,[0,T]}$, we can prove this by using the fact
$\Phi^N=L(Y^N)$, Lemma~\ref{Estimate for phi},
Lemma~\ref{Osc for YN} and Lemma~\ref{Garsia}.
We prove that
for any $a>0$, there exists $N_0$ such that
\begin{align}
\sup_{N\ge N_0}E[e^{a\|Y^N\|_{\infty,[0,T]}}]<\infty.\label{exponential
 integranility of YN}
\end{align}
By the estimate (\ref{estimate for YN}),
\begin{align}
\max_{0\le t\le T}|Y^N(t)|&\le \max_{0\le k\le N}|Y^N(t_k)|+C\max_{1\le
 k\le N}|\Delta B_k|+C/N.
\end{align}
Because $\sup_NE[e^{a\max_{1\le
 k\le N}|\Delta B_k|}]<\infty$  for all $a>0$,
it is sufficient to prove
\begin{align}
\sup_{N\ge N_0} E[e^{a\max_{0\le k\le N}|Y^N(t_k)|}]
\end{align}
By the decomposition and the estimates of $Y^N$ in Lemma~\ref{Estimate for YN},
we have
\begin{align}
\max_{0\le k\le N}|Y^N(t_k) |&
\le C+
\max_{0\le k\le N}|M^N(t_k)|+
C\sum_{k=1}^N|\Delta B_k|^2,
\end{align}
where $\{M^N(t)\}$ is the continuous martingale
which is defined in (\ref{martingale MN}).
Since the quadratic variation is bounded,
we have $\sup_NE[e^{a\max_{0\le t\le T}|M^N(t)|}]<\infty$
for any $a$.
Also
\begin{align}
 E\left[\exp\left(\sum_{k=1}^{N}Ca|\Delta B_k|^2\right)\right]&=
\prod_{k=1}^NE[e^{Ca|\Delta B_k|^2}]\nonumber\\
&=\prod_{k=1}^N\int_{\RR^n}\exp\left(\frac{CaT}{N}|x|^2-\frac{1}{2}|x|^2\right)
\frac{1}{(2\pi)^{n/2}}dx\nonumber\\
&=
\left(1-\frac{2CaT}{N}\right)^{-nN/2}\to e^{CaTn}\quad \mbox{as $N\to \infty$}.
\end{align}
These imply (\ref{exponential
 integranility of YN}) and the proof is finished.
\end{proof}

The following is a key lemma for the proof of
$L^p$ convergence of Wong-Zakai approximation.

\begin{lem}\label{Difference of wong-zakai and euler 1}
 Assume {\rm (A)}, {\rm (B)} and {\rm (C)}.
Let $X^N_E$ be the Euler-Peano approximation to 
SDE$(\sigma,\tilde{b})$,
where
$\tilde{b}=b+\frac{1}{2}\tr (D\sigma)(\sigma)$.
Then for any $0<\theta<1$,
there exists a positive constant $C_{\theta}$ 
such that for all $N$,
\begin{align}
\sup_{0\le k\le
 N}E\left[|X^N(t^N_k)-X^N_E(t^N_k)|^2\right]
\le C_{\theta}\cdot \Delta_N^{\theta/2}.\label{quarter}
\end{align}
\end{lem}

\begin{rem}
{\rm
The order of convergence in (\ref{quarter})
is, roughly speaking, half of that of
the Wong-Zakai approximation to the SDE without
reflection term.
This convergence order can be expected by
the $1/2$-H\"older continuity of the
Skorohod map.
Consider two Skorohod equations
$\xi=w+\phi$, $\xi'=w'+\phi'$.
Then it was proved in \cite{saisho} (see also \cite{lions-sznitman})
that under the assumptions (A) and (B),
\begin{align}
|\xi(t)-\xi'(t)|^2&\le
\left\{|w(t)-w'(t)|^2+
4\left(\|\phi\|_{[0,t]}+\|\phi'\|_{[0,t]}\right)
\max_{0\le s\le t}|w(s)-w'(s)|\right\}\nonumber\\
& \quad
\exp\left\{\left(\|\phi\|_{[0,t]}+\|\phi'\|_{[0,t]}\right)/r_0\right\},
\qquad 0\le t\le T.\label{Hoelder continuous map}
\end{align}
By this $1/2$-H\"older continuity of the Skorohod map $\Gamma$, we obtain
\begin{align}
 E[|\Gamma(B)_t-\Gamma(B^N)_t|^2]&\le
C\Delta_N^{\theta/2},\label{Gamma difference}
\end{align}
where $0<\theta<1$.
By examining the proof in \cite{saisho}, one can replace
the term $\|\phi\|_{[0,t]}+\|\phi'\|_{[0,t]}$ in 
(\ref{Hoelder continuous map}) by
$\|\phi-\phi'\|_{[0,t]}$.
We are not sure whether or not this change gives better estimates
than the above.
Of course the estimate in (\ref{Hoelder continuous map})
is a pathwise estimate and there are no reason that 
the pathwise estimate gives good estimate for the 
expectation also.
Of course, if $D$ is a half space (or convex polyhedron, see
 \cite{dupuis-ishii}) 
in a Euclidean space,
then $\Gamma$ is Lipschitz continuous and
the upper bound in (\ref{Gamma difference}) is $O(\Delta_N^{\theta})$.
Also, it seems that the calculation in \cite{doss}
also gives the convergence speed $O(\Delta_N^{\theta})$
for Wong-Zakai approximations of
general reflecting SDEs in the
half space case.
However, 
We do not know examples of reflecting SDE 
for which the slow convergence speed $\Delta_N^{\theta/2}$
really appear.
} 
\end{rem}

In the proof of this lemma, 
the integrals which contains ${\cal F}_t$-semimartingales
and non-adapted bounded variation processes,
{\it e.g.} Wong-Zakai approximation $X_N(t)$
appear.
Hence we need the following definition of
the integrals.

\begin{lem}\label{Extended stochastic integral}
Let $X(t), Y(t)$ be ${\cal F}_t$-continuous semimartingales and
$A(t)$ be bounded variation continuous process.
Suppose that $\sup_{0\le t\le T}\{|X(t)|+|Y(t)|\}+|A(\cdot)|_{[0,T]}\in L^p$
for all $p$.
Define
\begin{align}
 \int_0^tX(s)A(s)dY(s)&=
\lim_{N\to\infty}
\sum_{k=1}^NX(t^N_{k-1})A(t^N_{k-1})(Y(t^N_k)-Y(t^N_{k-1})),\\
\langle XA, Y\rangle_t&=\lim_{N\to\infty}
\sum_{k=1}^N\left(
(X(t^N_{k})A(t^N_{k}))-
(X(t^N_{k-1})A(t^N_{k-1}))\right)
(Y(t^N_k)-Y(t^N_{k-1})),
\end{align}
where $t^N_k=tk/N$.
These converge in probability 
and it holds that
\begin{align}
 \int_0^tX(s)A(s)dY(s)
&=\int_0^tA(s)dZ(s)=A(t)Z(t)-\int_0^tZ(s)dA(s),\label{mix integral}\\
\langle XA,Y\rangle_t&=
\int_0^tA(s)d\langle X, Y\rangle_s,
\end{align}
where 
$Z(s)=\int_0^tX(s)dY(s)$ is usual Ito integral
and the RHS of $(\ref{mix integral})$ is Riemann-Stieltjes
integral.
\end{lem}

Let us consider a set of stochastic processes ${\mathbb S}$
which consists of
a finite sum of 
product process $Y(t)A(t)$.
Here $Y(t)$ is a ${\cal F}_t$-
continuous semimartingale
and $A(t)$ is a continuous bounded variation process which is not
necessarily ${\cal F}_t$-adapted and
$\sup_{0\le t\le T}|Y(t)|+\|A\|_{[0,T]}\in \cap_{p\ge 1}L^p$.
Then this class is stable under the stochastic 
integral in the sense of the above lemma.
In the calculation below, we use the integrals of stochastic processes
in this sense.
Moreover the following chain rule holds.

\begin{lem}\label{Extended Ito formula}
 Let $Y,Z\in {\mathbb S}$.
Then
\begin{align}
 Y(t)Z(t)&=Y(0)Z(0)+\int_0^tY(s)dZ(s)+\int_0^tZ(s)dY(s)+
\langle Y,Z\rangle_t,
\end{align}
where $\langle Y,Z\rangle_t$
is defined similarly to Lemma~$\ref{Extended stochastic integral}$.
\end{lem}

The above two lemmas are proved by a standard argument
(integration by parts)
and we omit the proof.
In the proof of Lemma~\ref{Difference of wong-zakai and euler 1},
we use estimates on the expectations of the integrals in the above sense.
We introduce a family of iterated integrals.
Let ${\cal S}$ be a set of stochastic processes which consists of
the processes
$
 g(Y(t))
$
where $g$ is a $C^1$ function with values in $\RR$
with bounded derivative and
\begin{align}
 Y=X^N, X^N_E, B, B^N, \Phi^N(t), \Phi^N_E(t).
\end{align}
We define a set ${\cal S}_i$ of two parameter processes
$f=f(s,t)$~$(0\le s\le t\le T)$
inductively.
Let ${\cal S}_0=\{1\}$.
The set $S_i$ ($i\ge 1$) consists of
finite sums of
\begin{align}
 \prod_{k=1}^jf_k(s,t),\qquad
\int_s^tg(s,u)df_0(u),\label{f in Si}
\end{align}
where
$
 f_k\in S_{i_k}~\sum_{k=1}^ji_k=i~~i_k\ge 1
$
and
$
f_0\in {\cal S}, g\in {\cal S}_{i-1}.
$
Inductively, we see that
$f=f(s,t)\in {\cal S}_i$ is equal to a finite sum
of $g(s)h(t)$, where $g,h\in {\mathbb S}$.
Therefore, the integral in (\ref{f in Si})
is meaningful.
For these random variables, we have the following estimate.

\begin{lem}\label{Moment estimate for f}
Let $t^N_{k-1}\le s\le t\le t^N_k$.
Let $p\ge 1$.
For any $f\in {\cal S}_i$~$(i\in {\mathbb N})$,
there exists $C_p>0$ which is independent of $N,k$ such that
\begin{equation}
\|\max_{s\le u\le v\le t }f(u,v)\|_{L^p}\le C(p)(t-s)^{i/2}.
\end{equation}
\end{lem}

\begin{proof}
In this proof, we say that
 $f\in \cup_{i\ge 0}{\cal S}_i$
is adapted when the following holds.
The definition is given inductively by
\begin{itemize}
 \item[(i)] $1\in {\cal S}_0$ is adapted, 
\item[(ii)] Let $f$ be finite linear sums of processes
in (\ref{f in Si}).
Then $f$ is adapted if
all $f_k$~$(1\le k\le j)$ and $g$ are adapted
and $f_0=g(Y(t))$,
where $Y=X^N_E, B, \Phi^N_E$ and
$g$ is a $C^1$ function with bounded derivative.
\end{itemize}
By an induction on $i$, it is easy to check that
the set ${\cal S}_i$ 
is equal to the set of 
finite sums of two parameter processes
\begin{align}
\left(\prod_{k=1}^l\int_s^tg_k(s,u)dA^k(u)\right) \cdot
h(s,t),
\end{align}
where
\begin{itemize}
\item[(a)] $A^k$ is a bounded variation process in ${\cal S}$
and $g_k\in {\cal S}_{i_k}$.
When $l=0$, we set this product term as $1$.
\item[(b)] $h\in {\cal S}_{j}$ is adapted.
\item[(c)] the indices $i_k,j$ satisfy
$\sum_{k=1}^l(i_k+1)+j=i$.
\end{itemize}
Using this fact,
Lemma~\ref{Moment estimate for XNE},
Lemma~\ref{Moment estimate for XN},
Lemma~\ref{Estimate for YN} (\ref{estimate for YN})
and Lemma~\ref{estimate for xi},
we can complete the proof of the desired result by
an induction on $i$.
\end{proof}

\begin{proof}[Proof of Lemma~$\ref{Difference of wong-zakai and euler 1}$]
We write
\begin{align*}
Z^N(t)&=X^N_E(t)-X^N(t),\\
\rho_N(t)&=e^{-\frac{2}{\gamma}(f(X^N_E(t)+f(X^N(t))))},\\
 m_N(t)&=\rho_N(t)|Z^N(t)|^2.
\end{align*}
Let $t_{k-1}\le t\le t_k$.
By Lemma~\ref{Extended Ito formula},
\begin{align}
\lefteqn{d m_N(t)}\nonumber\\
&=
\rho_N(t)\Biggl\{
2\left(Z^N(t),\sigma(X^N_E(t_{k-1}))dB(t)-\sigma(X^N(t))
\frac{\Delta B_k}{\Delta}dt\right)\nonumber\\
& \quad +
2\left(Z^N(t),\tilde{b}(X^N_E(t_{k-1}))-b(X^N(t))\right)dt+
\tr\left[({}^t\sigma\sigma)(X^N_E(t_{k-1}))\right]dt\Biggr\}\nonumber\\
&\quad +
2\rho_N(t)\left(Z^N(t), d\Phi^N_E(t)-d\Phi^N(t)\right)\nonumber\\
&\quad
-\frac{2\rho_N(t)}{\gamma}
\left|Z^N(t)\right|^2\Bigl\{
\left((D f)(X^N_E(t)),d\Phi^N_E(t)\right)
+
\left((D f)(X^N(t)),d\Phi^N(t)\right)
\Bigr\}\nonumber\\
&\quad-\frac{2\rho_N(t)}{\gamma}
\left|Z^N(t)\right|^2\Bigl\{
\left((D f)(X^N_E(t)),\sigma(X^N_E(t_{k-1}))dB(t)\right)\nonumber\\
& \quad\qquad\qquad\qquad\qquad\qquad\qquad+
\left((D f)(X^N(t)),\sigma(X^N(t))\frac{\Delta B_k}{\Delta}dt\right)
\Bigr\}\nonumber\\
& \quad -\frac{4\rho_N(t)}{\gamma}
\sum_i\left((D f)(X^N_E(t)),\sigma(X^N_E(t_{k-1}))e_i\right)
\left(Z^N(t),\sigma(X^N_E(t_{k-1}))e_i\right)dt\nonumber\\
& \quad -\frac{2\rho_N(t)}{\gamma}|Z^N(t)|^2
\left\{\left((D f)(X^N_E(t)),\tilde{b}(X^N_E(t_{k-1}))\right)dt+
\left((D f)(X^N(t)),{b}(X^N(t))\right)dt
\right\}\nonumber\\
& \quad -\frac{\rho_N(t)}{\gamma}|Z^N(t)|^2\tr
 (D^2f)(X^N_E(t))
\left[
\sigma(X^N_E(t_{k-1})\cdot,\sigma(X^N_E(t_{k-1}))\cdot
\right]dt\nonumber\\
& \quad +\frac{2\rho_N(t)}{\gamma^2}|Z^N(t)|^2
|(D f)(X^N_E(t))(\sigma(X^N_E(t_{k-1})))|^2dt,
\end{align}
where $\{e_i\}$ is a c.o.n.s of
$\RR^n$.
After integrating both sides from $t_{k-1}$ to $t_k$, we
see that
the sum of the integral of the second term and the third term
on the RHS is non-positive by
the condition (C),
Therefore
\begin{align}
 m_N(t_k)&\le m_N(t_{k-1})+\sum_{k=1}^6I_k,
\end{align}
where
\begin{align*}
 I_1&=\int_{t_{k-1}}^{t_k}
\rho_N(t)\Biggl\{
2\left(Z^N(t),\sigma(X^N_E(t_{k-1}))dB(t)-\sigma(X^N(t))
\frac{\Delta B_k}{\Delta}dt\right)\nonumber\\
& \quad +
2\left(Z^N(t),\tilde{b}(X^N_E(t_{k-1}))-b(X^N(t))\right)dt
+\tr\left(({}^t\sigma\sigma)(X^N_E(t_{k-1}))\right)dt\Biggr\}
\end{align*}
\begin{align*}
 I_2&=-\int_{t_{k-1}}^{t_k}\frac{4\rho_N(t)}{\gamma}
\sum_i\left((D f)(X^N_E(t)),\sigma(X^N_E(t_{k-1}))e_i\right)
\left(Z^N(t),\sigma(X^N_E(t_{k-1}))e_i\right)dt
\end{align*}
\begin{align*}
 I_3&=
- \int_{t_{k-1}}^{t_k}\frac{2}{\gamma}m_N(t)
\Bigl\{
\left((D f)(X^N_E(t)),\sigma(X^N_E(t_{k-1}))dB(t)\right)\nonumber\\
& \quad\qquad\qquad\qquad\qquad\qquad\qquad+
\left((D f)(X^N(t)),\sigma(X^N(t))\frac{\Delta B_k}{\Delta}dt\right)
\Bigr\}
\end{align*}
\begin{align*}
 I_4&=
-\int_{t_{k-1}}^{t_k}\frac{2}{\gamma}m_N(t)
\left\{\left((D f)(X^N_E(t)),\tilde{b}(X^N_E(t_{k-1}))\right)dt+
\left((D f)(X^N(t)),b(X^N(t))\right)dt
\right\}
\end{align*}
\begin{align*}
 I_5&=
-\int_{t_{k-1}}^{t_k}\frac{m_N(t)}{\gamma}\tr
 (D^2f)(X^N_E(t))
\left[\sigma(X^N_E(t_{k-1})\cdot,\sigma(X^N_E(t_{k-1}))\cdot\right]dt
\end{align*}
\begin{align*}
 I_6&=\int_{t_{k-1}}^{t_k}\frac{2m_N(t)}{\gamma^2}
|(D f)(X^N_E(t))(\sigma(X^N_E(t_{k-1})))|^2dt.
\end{align*}
Let
$a_k=E[m_N(t_k)]$.
We prove that there exists a positive constant $C$ and $0<\theta<1$
which is independent
of $N$ and a non-negative sequence $\{b_k\}$ such that
\begin{align*}
 a_{k}&\le \left(1+\frac{CT}{N}\right)a_{k-1}+b_k\qquad 1\le k\le N\\
\sum_{k=1}^Nb_k&\le C\left(\frac{T}{N}\right)^{\theta/2}.
\end{align*}
Then we get
\begin{align*}
 a_k&\le
 \left(1+\frac{CT}{N}\right)^2a_{k-2}+\left(1+\frac{CT}{N}\right)b_{k-1}+b_k
\nonumber\\
&\le \left(1+\frac{CT}{N}\right)^ka_0+\sum_{i=0}^{k-1}
\left(1+\frac{CT}{N}\right)^ib_{k-i}\nonumber\\
&\le e^{CT}\sum_{i=1}^kb_i\le C_T\left(\frac{T}{N}\right)^{\theta/2}
\end{align*}
which is the desired estimate.
We consider
$I_k$ $(k=4,5,6)$.
By Lemma~\ref{Moment estimate for f},
we have 
$$
\|m_N(t)-m_N(t_{k-1})\|_{L^p}\le  C\Delta^{1/2}
\qquad t_{k-1}\le t\le t_k.
$$
Thus
\begin{align*}
 |E[I_k]|&\le
C\left(a_{k-1}\Delta+\Delta^{3/2}\right).
\end{align*}
So our task is to estimate $I_1,I_2,I_3$.
We consider $I_1$.
We rewrite
\begin{align*}
 I_1&=
J_1+J_2+J_3+J_4,
\end{align*}
where
\begin{align*}
 J_1&=
2\int_{t_{k-1}}^{t_k}
\rho_N(t_{k-1})
\Biggl\{
\left(Z^N(t),\sigma(X^N_E(t_{k-1}))dB(t)-\sigma(X^N(t_{k-1}))
\frac{\Delta B_k}{\Delta}dt\right)\nonumber\\
& \quad +
\frac{1}{2}\tr\left({}^t\sigma\sigma\right)(X^N_E(t_{k-1}))dt
\Biggr\}
\end{align*}
\begin{align*}
J_2&= 
2\int_{t_{k-1}}^{t_k}
\rho_N(t)\Bigl(Z^N(t),
\tilde{b}(X^N_E(t_{k-1}))-b(X^N(t))
-\left(\sigma(X^N(t))-\sigma(X^N(t_{k-1}))\right)
\frac{\Delta B_k}{\Delta}
\Bigr)dt,
\end{align*}
\begin{align*}
 J_3&=
2\int_{t_{k-1}}^{t_k}
(\rho_N(t)-\rho_N(t_{k-1}))
\left(Z^N(t),\sigma(X^N_E(t_{k-1}))dB(t)-\sigma(X^N(t_{k-1}))
\frac{\Delta B_k}{\Delta}dt\right),
\end{align*}
\begin{align*}
J_4&=
\int_{t_{k-1}}^{t_k}
\left(\rho_N(t)-\rho_N(t_{k-1})\right)
\tr\left({}^t\sigma\sigma\right)(X^N_E(t_{k-1}))dt.
\end{align*}
First we estimate $J_1$.
Let
\begin{align*}
 \tilde{J}_1&=
2\int_{t_{k-1}}^{t_k}
\rho_N(t_{k-1})
\Biggl\{
\left(Z^N(t)-Z^N(t_{k-1}),\sigma(X^N_E(t_{k-1}))dB(t)-\sigma(X^N(t_{k-1}))
\frac{\Delta B_k}{\Delta}dt\right)\nonumber\\
& \quad\quad +
\frac{1}{2}\tr\left({}^t\sigma\sigma\right)(X^N_E(t_{k-1}))dt
\Biggr\}.
\end{align*}
Then $E[J_1-\tilde{J}_1]=0$.
So it suffices to estimate the expectation of
$\tilde{J}_1$.
We rewrite 
\begin{align*}
 \tilde{J}_1&=
\sum_{k=1}^4\tilde{J}_{1,k},
\end{align*}
where
\begin{align*}
\tilde{J}_{1,1}&=
2\int_{t_{k-1}}^{t_k}
\rho_N(t_{k-1})
\Biggl\{
\Bigl(
\sigma(X^N_E(t_{k-1}))(B(t)-B(t_{k-1}))-\sigma(X^N(t_{k-1}))
\frac{\Delta B_k}{\Delta}(t-t_{k-1}),\nonumber\\
&\quad \quad\quad \sigma(X^N_E(t_{k-1}))dB(t)-\sigma(X^N(t_{k-1}))
\frac{\Delta B_k}{\Delta}dt\Bigr)\nonumber\\
& \quad\quad\quad +
\frac{1}{2}\tr\left({}^t\sigma\sigma\right)(X^N_E(t_{k-1}))dt
\Biggr\}.
\end{align*}
\begin{align*}
 \tilde{J}_{1,2}&=
-2\int_{t_{k-1}}^{t_k}
\rho_N(t_{k-1})
\Biggl\{
\Bigl(
\int_{t_{k-1}}^t\left(\sigma(X^N(s))-\sigma\left(X^N(t_{k-1})\right)\right)
\frac{\Delta B_k}{\Delta}ds,\nonumber\\
&\quad \sigma(X^N_E(t_{k-1}))dB(t)-\sigma(X^N(t_{k-1}))
\frac{\Delta B_k}{\Delta}dt\Bigr)\Biggr\}
\end{align*}
\begin{align*}
 \tilde{J}_{1,3}&=
2\int_{t_{k-1}}^{t_k}
\rho_N(t_{k-1})
\Biggl\{
\Bigl(\tilde{b}(X^N_E(t_{k-1})))(t-t_{k-1})-
\int_{t_{k-1}}^tb(X^N(s))ds,\nonumber\\
&\quad\quad\quad \sigma(X^N_E(t_{k-1}))dB(t)-\sigma(X^N(t_{k-1}))
\frac{\Delta B_k}{\Delta}dt\Bigr)\Biggr\}
\end{align*}
\begin{align*}
 \tilde{J}_{1,4}&=
2\int_{t_{k-1}}^{t_k}
\rho_N(t_{k-1})
\Biggl\{
\Bigl((\Phi^N_E(t)-\Phi^N_E(t_{k-1}))-
(\Phi^N(t)-\Phi^N(t_{k-1})),\nonumber\\
&\quad\quad \sigma(X^N_E(t_{k-1}))dB(t)-\sigma(X^N(t_{k-1}))
\frac{\Delta B_k}{\Delta}dt\Bigr)\Biggr\}\nonumber\\
\end{align*}
By a simple calculation,
\begin{align*}
 E[\tilde{J}_{1,1}]=
E\left[\rho_N(t_{k-1})
\|\sigma(X^N_E(t_{k-1}))-\sigma(X^N(t_{k-1}))\|_{H.S.}^2\right]
(t_k-t_{k-1})
\le C a_k\, \Delta.
\end{align*}
By Lemma~\ref{Moment estimate for f},
we have
\begin{align*}
 E[|\tilde{J}_{1,2}|]&\le
C\Delta^{3/2}.
\end{align*}
It is easy to see that
$E[|\tilde{J}_{1,3}|]\le C \Delta^{3/2}$.
By integrating by parts
\begin{align*}
|E[\tilde{J}_{1,4}]|&\le
C E\left[
\left(\|\Phi^N_E\|_{[t_{k-1},t_k]}+\|\Phi^N\|_{[t_{k-1},t_k]}\right)
\|B\|_{\infty,[t_{k-1},t_k]}\right]=:c_k.
\end{align*}
We have
\begin{align*}
 \sum_{k=1}^Nc_k&\le
C \|\left(\|\Phi^N_E\|_{[0,T]}+\|\Phi^N\|_{[0,T]}\right)\|_{L^2}
\|\max_k \|B\|_{\infty,[t_{k-1},t_k]}\|_{L^2}
\le C \Delta^{\theta/2}.
\end{align*}
We estimate $J_2$.
\begin{align*}
J_2&= 
2\int_{t_{k-1}}^{t_k}
\rho_N(t)\Bigl(Z^N(t),
\tilde{b}(X^N_E(t_{k-1}))-\tilde{b}(X^N(t_{k-1}))
+
b(X^N(t_{k-1}))
-b(X^N(t))\Bigr)dt\nonumber\\
&\quad+2\int_{t_{k-1}}^{t_k}
\rho_N(t)\Bigl(Z^N(t),
\tilde{b}(X^N(t_{k-1}))-b(X^N(t_{k-1}))\nonumber\\
& \qquad\qquad\qquad\qquad\quad
-\left(\sigma(X^N(t))-\sigma(X^N(t_{k-1}))\right)
\frac{\Delta B_k}{\Delta}
\Bigr)dt\nonumber\\
& =J_{2,1}+J_{2,2}.
\end{align*}
By rewriting $Z^N(t)=Z^N(t_{k-1})+Z^N(t)-Z^N(t_{k-1})$
and using Lemma~\ref{Moment estimate for f} and the Schwarz
inequality, we get
\begin{align*}
E[|J_{2,1}|]&\le
C(a_k\Delta+\Delta^{3/2}).
\end{align*}
We consider $J_{2,2}$.
Let
\begin{align*}
 \tilde{J}_{2,2}&=
2\int_{t_{k-1}}^{t_k}
\rho_N(t_{k-1})\Bigl(Z^N(t_{k-1}),
\tilde{b}(X^N(t_{k-1}))-b(X^N(t_{k-1}))\nonumber\\
& \qquad\qquad\qquad\qquad\quad
-\left(\sigma(X^N(t))-\sigma(X^N(t_{k-1}))\right)
\frac{\Delta B_k}{\Delta}
\Bigr)dt.
\end{align*}
By Lemma~\ref{Moment estimate for f},
we have
\begin{align*}
|E[J_{2,2}-\tilde{J}_{2,2}]|&\le
C\Delta^{3/2}.
\end{align*}
Noting
\begin{align*}
\lefteqn{
 \sigma(X^N(t))-\sigma(X^{N}(t_{k-1}))}\nonumber\\
&=\int_{t_{k-1}}^{t}(D\sigma)(X^N(s))
\left(\sigma(X^N(s))\frac{\Delta B_k}{\Delta}\right)ds+
\int_{t_{k-1}}^{t}(D\sigma)(X^N(s))(b(X^N(s)))ds\nonumber\\
& \quad+
\int_{t_{k-1}}^{t}(D\sigma)(X^N(s))(d\Phi^N(s))
\end{align*}
and for any $\xi$,
\begin{align*}
 E\left[\int_{t_{k-1}}^{t_k}(t-t_{k-1})(D\sigma)(\xi)
\left(\sigma(\xi)\frac{\Delta B_k}{\Delta}\right)
\left(\frac{\Delta B_k}{\Delta}\right)dt\right]&=
\frac{1}{2}\tr (D\sigma)(\xi)(\sigma(\xi))=(\tilde{b}-b)(\xi),
\end{align*}
we obtain
\begin{align*}
 |E[\tilde{J}_{2,2}]|&\le
CE\left[\|\Phi^N\|_{[t_{k-1},t_k]}\|B\|_{\infty,[t_{k-1},t_k]}\right]
+C\Delta^{3/2}.
\end{align*}
Therefore, as before, we get
$
|E[J_{2,2}]|\le
C\Delta^{3/2}+c_k,
$
where $c_k$ is a non-negative number such that
$\sum_{k=1}^Nc_k\le C\Delta^{\theta/2}$.
By Lemma~\ref{Moment estimate for f}, we have
$E[|J_4|]\le C\Delta^{3/2}$.
We estimate $J_3$ together with
$I_2$ and a term $I_{3,2}$ which is defined below.
We estimate $I_3$.
We rewrite
\begin{align*}
I_3&=I_{3,1}+I_{3,2}+I_{3,3}+I_{3,4}+I_{3,5}+I_{3,6}+I_{3,7},
\end{align*}
where
\begin{align*}
\lefteqn{I_{3,1}}\nonumber\\
&=-\frac{4}{\gamma}
\int_{t_{k-1}}^{t_k}
\rho_N(t)
\int_{t_{k-1}}^t\left(
Z^N(s)-Z^N(t_{k-1}),\sigma(X^N_E(t_{k-1}))dB(s)-\sigma(X^N(s))
\frac{\Delta B_k}{\Delta}ds
\right)\nonumber\\
& \times
\Biggl\{
\left((D f)(X^N_E(t)),\sigma(X^N_E(t_{k-1}))dB(t)\right)
+
\left((D f)(X^N(t)),\sigma(X^N(t))\frac{\Delta B_k}{\Delta}dt\right)
\Biggr\},
\end{align*}
\begin{align*}
I_{3,2}&=-\frac{4}{\gamma}
\int_{t_{k-1}}^{t_k}
\rho_N(t)
\int_{t_{k-1}}^t\left(
Z^N(t_{k-1}),\sigma(X^N_E(t_{k-1}))dB(s)-\sigma(X^N(s))
\frac{\Delta B_k}{\Delta}ds
\right)\nonumber\\
& \times
\Biggl\{
\left((D f)(X^N_E(t)),\sigma(X^N_E(t_{k-1}))dB(t)\right)
+
\left((D f)(X^N(t)),\sigma(X^N(t))\frac{\Delta B_k}{\Delta}dt\right)
\Biggr\},
\end{align*}
\begin{align*}
 I_{3,3}&=-\frac{4}{\gamma}\int_{t_{k-1}}^{t_k}
\rho_N(t)
\int_{t_{k-1}}^t\left(Z^N(s),
 \tilde{b}(X^N_E(t_{k-1}))-b(X^N(s))\right)ds
\nonumber\\
& \times
\Biggl\{
\left((D f)(X^N_E(t)),\sigma(X^N_E(t_{k-1}))dB(t)\right)
+
\left((D f)(X^N(t)),\sigma(X^N(t))\frac{\Delta B_k}{\Delta}dt\right)
\Biggr\},
\end{align*}
\begin{align*}
 I_{3,4}&=-\frac{4}{\gamma}\int_{t_{k-1}}^{t_k}
\rho_N(t)
\int_{t_{k-1}}^t\left(Z^N(s)-Z^N(t_{k-1}),
 d\Phi^N_E(s)-d\Phi^N(s)\right)
\nonumber\\
& \times
\Biggl\{
\left((D f)(X^N_E(t)),\sigma(X^N_E(t_{k-1}))dB(t)\right)
+
\left((D f)(X^N(t)),\sigma(X^N(t))\frac{\Delta B_k}{\Delta}dt\right)
\Biggr\},
\end{align*}
\begin{align*}
 I_{3,5}&=-\frac{4}{\gamma}\int_{t_{k-1}}^{t_k}
\rho_N(t)
\left(Z^N(t_{k-1}),
(\Phi^N_E(t)-\Phi^N_E(t_{k-1}))-
(\Phi^N(t)-\Phi^N(t_{k-1}))\right)
\nonumber\\
& \times
\Biggl\{
\left((D f)(X^N_E(t)),\sigma(X^N_E(t_{k-1}))dB(t)\right)
+
\left((D f)(X^N(t)),\sigma(X^N(t))\frac{\Delta B_k}{\Delta}dt\right)
\Biggr\},
\end{align*}
\begin{align*}
 I_{3,6}&=-\frac{2}{\gamma}
|Z^N(t_{k-1})|^2\int_{t_{k-1}}^{t_k}
\rho_N(t)
\Biggl\{
\left((D f)(X^N_E(t)),\sigma(X^N_E(t_{k-1}))dB(t)\right)\nonumber\\
&\quad +
\left((D f)(X^N(t)),\sigma(X^N(t))\frac{\Delta B_k}{\Delta}dt\right)
\Biggr\},\\
I_{3,7}&=-\frac{2}{\gamma}
\int_{t_{k-1}}^{t_k}
\rho_N(t)
(t-t_{k-1})\tr\left(({}^t\sigma\sigma)(X^N_E(t_{k-1}))\right)
\Biggl\{
\left((D f)(X^N_E(t)),\sigma(X^N_E(t_{k-1}))dB(t)\right)\nonumber\\
&\quad +
\left((D f)(X^N(t)),\sigma(X^N(t))\frac{\Delta B_k}{\Delta}dt\right)
\Biggr\}.
\end{align*}
As for $I_{3,1}, I_{3,3}, I_{3,4}, I_{3,7}$, by
Lemma~\ref{Moment estimate for f}, it is easy to
see
\begin{align*}
|E[I_{3,k}]|&\le 
C\Delta^{3/2}.
\end{align*}
We consider $I_{3,6}$.
By (\ref{estimate for YN}) and Lemma~\ref{estimate for xi},
we have 
$$
\|X^N\|_{[t_{k-1},t_k]}\le C\left(|\Delta B_k|+\Delta\right).
$$
Using this estimate and (\ref{estimate for xe}),
we have
\begin{multline*}
\Biggl|E\Biggl[
\int_{t_{k-1}}^{t_k}
\rho_N(t)
\Biggl\{
\left((D f)(X^N_E(t)),\sigma(X^N_E(t_{k-1}))dB(t)\right)\nonumber\\
+\left((D f)(X^N(t)),\sigma(X^N(t))\frac{\Delta B_k}{\Delta}dt\right)
\Biggr\} \Bigg |{\cal F}_{t_{k-1}}
\Biggr] 
\Biggr|\le
C\Delta.\phantom{lllllllllllllllllll}
\end{multline*}
Hence,
$|E[I_{3,6}]|\le C a_{k-1}\Delta$.
Using Lemma~\ref{Moment estimate for f},
we have there exists non-negative random variable $I'_{3,5}$ such that
$E[I'_{3,5}]\le C\Delta^{3/2}$ and
\begin{align*}
|I_{3,5}|&\le 
C|Z^N(t_{k-1})|G_k+I'_{3,5}\le
\frac{C}{2}\left(|Z^N(t_{k-1})|^2G_k+G_k\right)+I'_{3,5},
\end{align*}
where
\begin{align*}
 G_k&=\left(\max_{t_{k-1}\le t\le t_k}
\left|\int_{t_{k-1}}^t
\left((D f)(X^N_E(s)),\sigma(X^N_E(t_{k-1}))dB(s)\right)
\right|+\|B\|_{\infty,[t_{k-1},t_k]}\right)\nonumber\\
&\quad
\times \left(\|\Phi^N_E\|_{[t_{k-1},t_k]}+
\|\Phi^N\|_{[t_{k-1},t_k]}\right).
\end{align*}
Using $E[G_k | {\cal F}_{t_{k-1}}]\le C\Delta$,
we obtain
\begin{align*}
E[|I_{3,5}|] &\le
C(a_{k-1}\Delta+E\left[G_k\right])
+C\Delta^{3/2}.
\end{align*}
Since
\begin{align*}
\sum_{k=1}^NE[G_k]&\le
E\Bigg[
\left(\|\Phi^N_E\|_{[0,T]}+
\|\Phi^N\|_{[0,T]}\right)\nonumber\\
&\quad \times
\max_k
\left(\max_{t_{k-1}\le t\le t_k}
\left|\int_{t_{k-1}}^t
\left((D f)(X^N_E(s)),\sigma(X^N_E(t_{k-1}))dB(s)\right)
\right|+\|B\|_{\infty,[t_{k-1},t_k]}\right)
\Biggr]\nonumber\\
&\le C\Delta^{\theta/2},
\end{align*}
we obtain the desired estimate for $I_{3,5}$.
Finally we estimate $I_2+J_3+I_{3,2}$.
First we rewrite $J_3$.
Note that
\begin{align*}
\rho_N(t)-\rho_N(t_{k-1})
&=
-\frac{2}{\gamma}\rho_N(t_{k-1})
\Biggl\{\Bigl((D f)(X^N_E(t_{k-1})), 
\sigma(X^N_E(t_{k-1}))(B(t)-B(t_{k-1}))\Bigr)\nonumber\\
& \quad +
\left((D f)(X^N(t_{k-1})), 
\sigma(X^N(t_{k-1}))
\frac{\Delta B_k}{\Delta}(t-t_{k-1})\right)
\Biggr\}\nonumber\\
&\quad -\frac{2}{\gamma}\rho_N(t_{k-1})
\Biggl\{\Bigl((D f)(X^N_E(t_{k-1})), 
\tilde{b}(X^N_E(t_{k-1}))(t-t_{k-1})\Bigr)\nonumber\\
& \quad +
\left((D f)(X^N(t_{k-1})), 
b(X^N(t_{k-1}))(t-t_{k-1})\right)
\Biggr\}\nonumber\\
& -\frac{2}{\gamma}\rho_N(t_{k-1})
\Biggl\{\Bigl((D f)(X^N_E(t_{k-1})), 
\Phi^N_E(t)-\Phi^N_E(t_{k-1})\Bigr)\nonumber\\
& \quad +
\left((D f)(X^N(t_{k-1})), 
\Phi^N(t)-\Phi^N(t_{k-1})\right)
\Biggr\}+\tilde{\rho}(t_{k-1},t).
\end{align*}
Here $\tilde{\rho}\in {\cal S}_2$.
Hence we can neglect the term $\tilde{\rho}$
to estimate $J_3$ by Lemma~\ref{Moment estimate for f}.
Also we can estimate the terms containing
$\Phi^N_E, \Phi^N$ in a similar way
to $\tilde{J}_{1,4}, \tilde{J}_{2,2}, I_{3,5}$.
We can estimate the term containing $b, \tilde{b}$ 
by Lemma~\ref{Moment estimate for f}.
Consequently, we can replace the term
$J_3$ by $\tilde{J}_3$:
\begin{align*}
\tilde{J}_3&= 
-\frac{4}{\gamma}
\int_{t_{k-1}}^{t_k}\rho_N(t_{k-1})
\Biggl\{\Bigl((D f)(X^N_E(t_{k-1})), 
\sigma(X^N_E(t_{k-1}))(B(t)-B(t_{k-1}))\Bigr)\nonumber\\
& \quad +
\left((D f)(X^N(t_{k-1})), 
\sigma(X^N(t_{k-1}))
\frac{\Delta B_k}{\Delta}(t-t_{k-1})\right)
\Biggr\}\nonumber\\
&\quad \times \left(Z^N(t_{k-1}),\sigma(X^N_E(t_{k-1}))dB(t)-\sigma(X^N(t_{k-1}))
\frac{\Delta B_k}{\Delta}dt\right).
\end{align*}
Also,
similarly,
we can replace
$I_{3,2}$ by $\tilde{I}_{3,2}$:
\begin{align*}
 \tilde{I}_{3,2}&=
-\frac{4}{\gamma}
\int_{t_{k-1}}^{t_k}
\rho_N(t_{k-1})\nonumber\\
& \quad \times
\left(
Z^N(t_{k-1}),\sigma(X^N_E(t_{k-1}))(B(t)-B(t_{k-1}))-\sigma(X^N(t_{k-1}))
\frac{\Delta B_k}{\Delta}(t-t_{k-1})
\right)\nonumber\\
& \times
\Biggl\{
\left((D f)(X^N_E(t_{k-1})),\sigma(X^N_E(t_{k-1}))dB(t)\right)
+
\left((D f)(X^N(t_{k-1})),\sigma(X^N(t_{k-1}))\frac{\Delta B_k}{\Delta}dt\right)
\Biggr\}.
\end{align*}
By a simple calculation, we have
\begin{align*}
 \lefteqn{E\left[\tilde{J}_3~|~{\cal F}_{t_{k-1}}\right]}\nonumber\\
&=
\frac{2\Delta}{\gamma}\rho_N(t_{k-1})
\sum_{i}\left((D f)
(X^N_E(t_{k-1})),\sigma(X^N_E(t_{k-1}))e_i\right)
\left(Z^N(t_{k-1}),\sigma(X^N(t_{k-1}))e_i\right)\nonumber\\
&\quad -\frac{2\Delta}{\gamma}\rho_N(t_{k-1})
\sum_{i}\left((D f)
(X^N(t_{k-1})),\sigma(X^N(t_{k-1}))e_i\right)
\left(Z^N(t_{k-1}),\sigma(X^N_E(t_{k-1}))e_i\right)\nonumber\\
&\quad +\frac{2\Delta}{\gamma}\rho_N(t_{k-1})
\sum_{i}\left((D f)
(X^N(t_{k-1})),\sigma(X^N(t_{k-1}))e_i\right)
\left(Z^N(t_{k-1}),\sigma(X^N(t_{k-1}))e_i\right),
\end{align*}
\begin{align*}
 \lefteqn{E\left[\tilde{I}_{3,2}~|~{\cal F}_{t_{k-1}}\right]}\nonumber\\
&=
-\frac{2\Delta}{\gamma}\rho_N(t_{k-1})\sum_i
\left(Z^N(t_{k-1}),\sigma(X^N_E(t_{k-1}))e_i\right)
\left((D f)(X^N(t_{k-1})),\sigma(X^N(t_{k-1}))e_i\right)\nonumber\\
&\quad +\frac{2\Delta}{\gamma}\rho_N(t_{k-1})\sum_i
\left(Z^N(t_{k-1}),\sigma(X^N(t_{k-1}))e_i\right)
\left((D
 f)(X^N_E(t_{k-1})),\sigma(X^N_E(t_{k-1}))e_i\right)\nonumber\\
&\quad +\frac{2\Delta}{\gamma}\rho_N(t_{k-1})\sum_i
\left(Z^N(t_{k-1}),\sigma(X^N(t_{k-1}))e_i\right)
\left((D
 f)(X^N(t_{k-1})),\sigma(X^N(t_{k-1}))e_i\right).
\end{align*}
Hence
\begin{align*}
\lefteqn{E\left[\tilde{J}_3+\tilde{I}_{3,2}\right]}\nonumber\\
&=Ca_{k-1}\Delta\nonumber\\
&\quad+\frac{4\Delta}{\gamma}E\left[
\rho_N(t_{k-1})
\sum_i
\left(Z^N(t_{k-1}),\sigma(X^N(t_{k-1}))e_i\right)
\left((D
 f)(X^N(t_{k-1})),\sigma(X^N(t_{k-1}))e_i\right)
\right].
\end{align*}
Consequently,
\begin{align*}
\left|E\left[I_2+\tilde{J}_3+\tilde{I}_{3,2}\right]\right|&\le
C\left(a_{k-1}\Delta+\Delta^{3/2}\right).
\end{align*}
This completes the proof.
\end{proof}

\begin{rem}{\rm
 Note that some parts in the above estimate for $a_k=E[m_n(t_k)]$ are rough.
In the case where $\partial D=\emptyset$, that is, 
$D=\RR^d$, the local time terms
vanish.
In this case, the estimate
\begin{align}
 a_k&\le (1+C\Delta)a_{k-1}+C \Delta^2
\end{align}
might be true.
The bad term $\Delta^{\theta/2}$ essentially 
comes from the estimates on local time terms if
$\partial D\ne \emptyset$.}
\end{rem}

\begin{lem}\label{Difference of wong-zakai and euler 2}
Assume the same assumptions in Lemma~$\ref{Difference of wong-zakai and
euler 1}$ and consider the
same SDE. Let $0<\theta<1$.
Then there exists a positive constant $C_{p,T,\theta}$ such that
\begin{align}
 E\left[\max_{0\le t\le T}|X^N(t)-X^N_E(t)|^{2p}\right]\le
C_{p,T,\theta}\Delta_N^{\theta/6}.
\end{align}
\end{lem}

\begin{proof}
Let $N_0$ be a natural number and 
choose a sufficiently large natural number $N$.
Pick partition points
$\{s_k\}_{k=1}^{N_0}\subset \{t^N_k\}_{k=1}^N$ such that
\begin{align*}
 |s_k-t^{N_0}_k|\le \frac{T}{N}.
\end{align*}
Let $t^{N_0}_{k-1}\le t\le t^{N_0}_k$.
Then
\begin{align*}
 |X^N(t)-X^N_E(t)|&\le
|X^N(t)-X^N(t^{N_0}_k)|+|X^N(t^{N_0}_{k})-X^N(s_k)|+
|X^N(s_k)-X^N_E(s_k)|\nonumber\\
&\quad +
|X^N_E(s_k)-X^N_E(t^{N_0}_k)|+|X^N_E(t^{N_0}_k)-X^N_E(t)|
\end{align*}
and
\begin{align*}
 \max_{0\le t\le T}|X^N(t)-X^N_E(t)|&\le
2\max_{0\le s\le t\le T, |t-s|\le T/N_0}|X^N(t)-X^N(s)|\nonumber\\
&\quad + 2\max_{0\le s\le t\le T, |t-s|\le T/N_0} |X^N_E(t)-X^N_E(s)|\nonumber\\
&\quad +\sum_{k=1}^{N_0}|X^N(s_k)-X^N_E(s_k)|.
\end{align*}
Therefore
\begin{align*}
 E\left[\max_{0\le t\le T}|X^N(t)-X^N_E(t)|^{2p}\right]&
\le C_{p,\theta}3^{2p-1}2^{2p}\left(\frac{T}{N_0}\right)^{p\theta}+
N_0^{2p-1}\sum_{k=1}^{N_0}E\left[|X^N(s_k)-X^N_E(s_k)|^{2p}\right]\nonumber\\
& \le C_{p,\theta}\left(3^{2p-1}2^{2p+1}\left(\frac{T}{N_0}\right)^{p\theta}+
N_0^{2p}\left(\frac{T}{N}\right)^{\theta/2}\right).
\end{align*}
Here we use the uniform moment estimate for $X^N_E, X^N$ and
Lemma~\ref{Difference of wong-zakai and euler 1} and
Lemma~\ref{Garsia}.
Hence setting $N_0$ as the integer part of $N^{1/6p}$,
we obtain the desired estimate.
\end{proof}

\begin{proof}[Proof of main theorem]
 The proof follows from 
Theorem~\ref{Euler-Peano approximation} and
Lemma~\ref{Difference of wong-zakai and euler 2}.
\end{proof}

\noindent
{\bf Acknowledgement}

\noindent
The authors would like to thank the referee for the 
valuable comments and suggestions which improve the 
quality of the paper.

\end{document}